\documentclass[12pt]{amsart}
 \usepackage{amsmath, amsthm, amsfonts, amssymb, color}
\usepackage{amsfonts, amsmath}
\begin{document}

\theoremstyle{plain}
\newtheorem{theorem}{Theorem}[section]
\newtheorem{prop}[theorem]{Proposition}
\newtheorem{lemma}[theorem]{Lemma}
\newtheorem{corollary}[theorem]{Corollary}
\newtheorem{example}[theorem]{Example}
\newtheorem{remark}[theorem]{Remark}
\newcommand{\ra}{\rightarrow}
\renewcommand{\theequation}
{\thesection.\arabic{equation}}
\newcommand{\ccc}{{\cal C}}

\theoremstyle{definition}
\newtheorem{definition}[theorem]{Definition}

\allowdisplaybreaks

\def\HAL { H^1_{{at}, L_1, L_2, M}({\Bbb R}^{n_1}\times{\Bbb R}^{n_2})}
 \def\HSL { H^1_{L_1, L_2, S}({\Bbb R}^{n_1}\times{\Bbb R}^{n_2}) }
\def\HL  {   H^1_{L_1, L_2}({\Bbb R}^{n_1}\times{\Bbb R}^{n_2}) }
\def\HLD  {  H^{\rm 1}_{L} \cap L^2   }
\def\RR{\mathbb R}
\def\Rn{{\mathbb R}^n}
\def\Rm{{\mathbb R}^m}

 \medskip

\title  [   singular integrals with non-smooth kernels]
{ End-point estimates for singular integrals  with non-smooth kernels   on product  spaces }

\author{Xuan Thinh Duong,\  \ Ji Li \   and \   Lixin Yan}
\address {Xuan Thinh Duong, Department of Mathematics, Macquarie University, NSW 2109, Australia}
\email{xuan.duong@mq.edu.au}
\address{
Ji Li, Department of Mathematics, Macquarie University, NSW 2109, Australia}
\email{
ji.li@mq.edu.au}
\address{
Lixin Yan, Department of Mathematics, Sun Yat-sen (Zhongshan) University, Guangzhou, 510275, P.R. China}
\email{
mcsylx@mail.sysu.edu.cn
}
 \footnotetext[1]{{\it {\rm 2010} Mathematics Subject Classification:}
42B20, 42B25, 42B30.}
\footnotetext[1]{{\it Key words and phrase:}  Singular integrals, Hardy spaces, product space,
 atomic decomposition, spectral multiplier theorem.}

\date{}


\maketitle

\medskip

\begin{abstract} The main aim of this article is to   establish
boundedness of singular integrals
with non-smooth kernels on product spaces.
Let $L_1$ and $L_2$ be non-negative self-adjoint operators
on $L^2(\mathbb{R}^{n_1})$ and $L^2(\mathbb{R}^{n_2})$, respectively, whose heat kernels
satisfy Gaussian upper bounds. First, we obtain an atomic decomposition
for functions in $H^1_{L_1,
L_2}(\mathbb{R}^{n_1}\times\mathbb{R}^{n_2})$
where the Hardy space $H^1_{L_1,
L_2}(\mathbb{R}^{n_1}\times\mathbb{R}^{n_2})$ associated with $L_1$ and $L_2$
is defined by
square function norms, then prove an interpolation
property for this space. 
Next,
we establish  sufficient conditions for certain singular integral operators
to be bounded on   
the Hardy space $H^1_{L_1,
L_2}(\mathbb{R}^{n_1}\times\mathbb{R}^{n_2})$
when the associated kernels of these singular integrals only satisfy
regularity conditions significantly weaker than
those of the standard Calder\'on--Zygmund kernels.
  As applications,
we  obtain  endpoint  estimates  of the double Riesz transforms  associated to Schr\"odinger operators
 and a Marcinkiewicz-type
  spectral multiplier theorem  for non-negative self-adjoint operators on product spaces.

\end{abstract}

\maketitle

\tableofcontents

\section{Introduction and main results}
\setcounter{equation}{0}

Modern harmonic analysis was introduced in the 50s with the
Calder\'on--Zygmund theory at the heart of it. This theory
established criteria for singular integral operators to be
bounded on different scales of function spaces, especially the
Lebesgue spaces $L^p$, $1 < p < \infty$. To achieve this goal,
an integrated part of the Calder\'on--Zygmund theory includes
the theory of interpolation and the theory of function spaces,
in particular end-point spaces such as the Hardy and BMO spaces.

While the Calder\'on--Zygmund theory with one parameter was well
established in the three decades of the 60s to 80s, multiparameter
Fourier analysis was introduced later in the 70s and studied
extensively in the 80s by a number of well known mathematicians,
including R. Fefferman, S.-Y. A. Chang, R. Gundy, E. Stein, J.L.
Journ\'e (see for instance, \cite {CF1}, \cite{CF2}, \cite{CF3},
\cite{F1}, \cite{F2}, \cite{F3}, \cite{F4}, \cite{FSt}, \cite{GS},
\cite{Jo}). However, in contrast to the established one-parameter
theory, the multiparameter theory is much more complicated and is
less advanced, especially that  there was not much progress in the
last two decades on the topic of singular integrals with non-smooth
kernels on product spaces.

 The aim of this article is twofold: to obtain an atomic decomposition for the  new Hardy spaces
 introduced recently in \cite{DSTY}  and to establish
end-point estimates for singular integrals
with non-smooth kernels on product spaces.

Let us remind the reader that in the standard theory of singular integrals on product domains, R. Fefferman
 obtained the boundedness properties on Hardy spaces $H^1(\mathbb{R}^{n_1}\times\mathbb{R}^{n_2})$ and on $L ({\rm log}^+ L)$
 of singular integrals that generalize the double Hilbert transform on
 product domains (\cite{F1}) as follows.
Suppose that $T$ is a bounded linear operator
 on $L^2(\mathbb{R}^{n_1}\times\mathbb{R}^{n_2})$ with an associated kernel $K(x_1,y_1,x_2,y_2)$ in the sense that

\begin{eqnarray}\label{def of singular integral operator}
Tf(x_1,x_2)=\iint_{\mathbb{R}^{n_1}\times\mathbb{R}^{n_2}}
K(x_1,y_1,x_2,y_2)f(y_1,y_2)dy_1dy_2,
\end{eqnarray}

\noindent
and the above formula holds for each continuous function $f$ with compact support, and for almost all $(x_1, x_2)$
not in the support of $f$. For each $x_1, y_1\in \mathbb{R}^{n_1}$, set
${\widetilde K}^{(1)} (x_1, y_1) (x_2, y_2)$ to be the integral
operator acting on functions one variable whose kernel is given by

\begin{eqnarray} \label{0}
{\widetilde K}^{(1)} (x_1, y_1) (x_2, y_2)=K(x_1,y_1,x_2,y_2).
\end{eqnarray}

\noindent
Similarly, let

\begin{eqnarray} \label{1}
{\widetilde K}^{(2)}(x_2, y_2) (x_1, y_1)=K(x_1,y_1,x_2,y_2).
\end{eqnarray}

For an integral operator, $s$, acting on a function $f\in L^2({\mathbb R}^n)$
and given by

$$
sf(x)=\int_{{\mathbb R}^n} k(x,y)f(y)dy,
$$

\noindent
if $$ (\ast)\ \ \int_{|x-y|>\gamma |y-y'|}
 |k(x,y)- k(x,y') | dx\leq
  C\gamma^{-\delta}$$
for $\gamma\geq 2$ and some $\delta>0,$
 then define
$|s|_{CZ}=\|s\|_{L^2, L^2} +\inf \{C>0|\ (\ast) \ {\rm holds}\}.$
With these conventions, we can state a result of R. Fefferman
(\cite{F1}):

\medskip

 \noindent {\bf Theorem A.} {\it
Let $T$ be a bounded linear operator
 on $L^2(\mathbb{R}^{n_1}\times\mathbb{R}^{n_2})$ with an associated kernel $K(x_1,y_1,x_2,y_2)$.
Suppose that   there exists some constant  $\delta>0$   such
that for all $\gamma \geq 2$,

\begin{eqnarray} \label{e1.4}
 \int_{|x_1-y_1|>\gamma |y_1-y'_1|}
\big|{\widetilde K}^{(1)}{(x_1,y_1)}- {\widetilde
K}^{(1)}{(x_1,y'_1)}\big|_{ CZ} dx_1\leq C\gamma^{-\delta}
\end{eqnarray}
and
\begin{eqnarray} \label{e1.5}
 \int_{|x_2-y_2|>\gamma |y_2-y'_2|}
\big|{\widetilde K}^{(2)}{(x_2,y_2)}- {\widetilde
K}^{(2)}{(x_2,y'_2)}\big|_{ CZ}dx_2\leq C\gamma^{-\delta}.
\end{eqnarray}

\noindent
Then $T$ extends to a bounded operator from
$H^1(\mathbb{R}^{n_1}\times\mathbb{R}^{n_2})$ to
$L^1(\mathbb{R}^{n_1}\times\mathbb{R}^{n_2})$, and also has the following weak type estimate on
$L ({\rm log}^+ L)$:

$$
\big| \big\{ |Tf(x)|>\alpha, (x_1, x_2) \in [0,1]^{n_1}\times [0,1]^{n_2}\big\}\big| \leq {C\over \alpha}\|f\|_{L ({\rm log}^+ L)}, \ \ \ \forall \alpha>0
$$

\noindent
for all functions $f(x_1,x_2)$  whose supports are contained in the unit square.
 }

\bigskip

It should be noted  that unlike the one  parameter case, the operator $T$ does not satisfy
the weak type $(1,1)$ estimate. That is, for every $f\in L^1(\mathbb{R}^{n_1}\times\mathbb{R}^{n_2}),$
endpoint estimate  of

$$
\big| \big\{(x_1, x_2)\in \mathbb{R}^{n_1}\times\mathbb{R}^{n_2}:  |Tf(x)|>\alpha\big\}\big|
\leq {C\over \alpha}\|f\|_{L^1(\mathbb{R}^{n_1}\times\mathbb{R}^{n_2})} , \ \ \ \forall \alpha>0
$$

\noindent
fails. From the point of view of interpolation theory,
Theorem A shows that the  Calder\'on--Zygmund   theory on product domains  shift the focus of the attention from
the `weak' $L^1$ theory to the `strong' $(H^1, L^1)$-theory (see for examples, \cite{CF1, CF2, CF3, F1, F2, F3, F4}).

In this  article, we consider certain singular integrals on product spaces whose kernels are not smooth enough
to fall under the scope of Theorem A. More specifically, we replace
the smoothness conditions (\ref{e1.4}),  (\ref{e1.5}) by the weaker conditions (\ref{condn1-}), (\ref{condn2-}), and we add condition
(\ref {condn3-}). To overcome the difficulties created from the absence of
(\ref{e1.4}) and (\ref{e1.5}), our strategy is to use suitable generalized families of approximations to the identity
as in \cite{DM} and to develop atomic decomposition for suitable Hardy spaces associated with operators
(see for example \cite{DY1} and \cite{DSTY} for related Hardy spaces associated with operators).
 We then establish   endpoint  estimates of those singular integrals with non-smooth kernels
on appropriate Hardy spaces $H^1_{L_1,
L_2}(\mathbb{R}^{n_1}\times\mathbb{R}^{n_2})$,
where
  $H^1_{L_1,
L_2}(\mathbb{R}^{n_1}\times\mathbb{R}^{n_2})$  is a  class of Hardy spaces  associated with
non-negative self-adjoint operators  with Gaussian  upper bounds on
theirs heat kernels. See Section 3  for a detailed study of these Hardy spaces. We note that the need to study
Hardy spaces associated with operators arises from the fact that singular integral operators with non-smooth kernels
might not behave well on the standard Hardy spaces.

Our framework is as follows. Let $T$ be a bounded linear operator
 on $L^2(\mathbb{R}^{n_1}\times\mathbb{R}^{n_2})$ with an associated kernel $K(x_1,y_1,x_2,y_2)$.
Let $L_i, i=1,2$ be non-negative self-adjoint operators on
$L^2({\mathbb R}^{n_i})$ and that the semigroup $e^{-tL_i}$,
generated by $-L_i$ on $L^2({\mathbb R}^{n_i})$,  has the kernel
$p_{t}(x_i,y_i)$ which  satisfies the following  Gaussian upper
bound

$$
|p_{t}(x_i,y_i)| \leq \frac{C}{t^{n_i/2} } \exp\Big(-{|x_i-y_i|^2\over
c\,t}\Big), \ \ t>0, \ \ x_i, y_i\in{\mathbb R}^{n_i}.
\leqno{\rm(GE)}
$$

 \noindent
 Consider  the composite operators
$T\circ(e^{-t_1L_1}\otimes e^{-t_2L_2}), t_i\geq 0,$ which  have associated kernels
  $K_{(t_1, t_2)}(x_1,y_1,x_2,y_2)$ in the sense of (\ref{def of singular integral operator}).
For convenience, we write

\begin{eqnarray*}
&&\hspace{-1cm}\Delta K_{(t_1,\, t_2)}(x_1,y_1,x_2,y_2)\\[3pt]
&=& \big|K(x_1,y_1,x_2,y_2)-K_{(t_1,\, 0)}(x_1,y_1,x_2,y_2)-K_{(0,\, t_2)}(x_1,y_1,x_2,y_2)
  +K_{(t_1,\, t_2)}(x_1,y_1,x_2,y_2)\big|.
\end{eqnarray*}

\noindent
Set

$$
{\widetilde K}^{(1)}_{(t_1, t_2)} (x_1, y_1) (x_2, y_2)=K_{(t_1, t_2)}(x_1,y_1,x_2,y_2),\\[2pt]
$$
$$
{\widetilde K}^{(2)}_{(t_1, t_2)} (x_2, y_2) (x_1, y_1)=K_{(t_1, t_2)}(x_1,y_1,x_2,y_2).
$$

Instead of conditions (\ref{e1.4}) and (\ref{e1.5}), we assume the
following: Suppose that the composite operators
$T\circ(e^{-t_1L_1}\otimes e^{-t_2L_2}), t_i\geq 0, i=1,2$ have
associated kernels
  $K_{(t_1, t_2)}(x_1,y_1,x_2,y_2)$ in the sense of (\ref{def of singular integral operator})
  and there exist some constants $\delta>0$ and $C>0$ such
that for all $\gamma_1,\gamma_2\geq 2$,

\begin{eqnarray} \label{condn1-}
 \int_{|x_1-y_1|>\gamma_1t_1}
\|{\widetilde K}^{(1)}{(x_1,y_1)}-{{\widetilde K}_{(t_1^2, \, 0)}^{(1)}}{(x_1,y_1)}\|_{ (L^2({\mathbb R}^{n_2}),\, L^2({\mathbb R}^{n_2}) )}dx_1\leq
C\gamma_1^{-\delta},
\end{eqnarray}

\begin{eqnarray}\label{condn2-}
 \int_{|x_2-y_2|>\gamma_2t_2}
\|{\widetilde K}^{(2)}{(x_2,y_2)}-{{\widetilde K}_{(0, \, t_2^2)}^{(2)}}{(x_2,y_2)}\|_{ (L^2({\mathbb R}^{n_1}),\, L^2({\mathbb R}^{n_1}) )}dx_2\leq
C\gamma_2^{-\delta},
\end{eqnarray}
\begin{eqnarray} \label{condn3-}
 \int_{\substack{|x_1-y_1|>\gamma_1t_1\\ |x_2-y_2|>\gamma_2t_2}}\big|\Delta K_{(t_1^2,t_2^2)}(x_1,y_1,x_2,y_2)\big|dx_1dx_2\leq
C\gamma_1^{-\delta}\gamma_2^{-\delta}.
\end{eqnarray}

We note that in our conditions (\ref{condn1-}) and (\ref{condn2-}), the $L^2$ norm of the operators
in the integrands were used in place of the CZ estimate (which is stronger than $L^2$ norm) in (\ref{e1.4}) and (\ref{e1.5}),
meanwhile condition (\ref{condn3-}) plays a similar role to the required CZ estimate in \cite{F1}. It can be checked that
 conditions (\ref{condn1-}) and (\ref{condn2-}) are indeed weaker than conditions (\ref{e1.4}) and (\ref{e1.5}) (see \cite{DM}).
  Our main result on the boundedness of singular integrals is the following (Theorem \ref{theorem T from Hp to Lp} in Section 4).





 \bigskip
{\it Let $T$ be a bounded linear operator
 on $L^2(\mathbb{R}^{n_1}\times\mathbb{R}^{n_2})$ which
satisfies conditions   \eqref{condn1-},  \eqref{condn2-} and  \eqref{condn3-}. Then
%
%
$T$ extends to a bounded operator from
$H_{L_1, L_2}^1(\mathbb{R}^{n_1}\times\mathbb{R}^{n_2})$ to
$L^1(\mathbb{R}^{n_1}\times\mathbb{R}^{n_2})$,
hence by interpolation, bounded from $L^p(\mathbb{R}^{n_1}\times\mathbb{R}^{n_2})$
to $L^p(\mathbb{R}^{n_1}\times\mathbb{R}^{n_2})$ for $1 < p < 2$.}
\smallskip


\bigskip

In Section 6, we shall exhibit  a class of singular integrals which satisfy the conditions
 (\ref{condn1-}),  (\ref{condn2-}) and  (\ref{condn3-}) but not the conditions (\ref{e1.4}) and (\ref{e1.5}).
More specifically, we use Theorem
\ref{theorem T from Hp to Lp} 
to obtain boundedness of

\medskip

 (i)  the double Riesz transforms  associated to Schr\"odinger operators with non-negative potentials
 (Theorem \ref{th6.0}); and

 \medskip

 (ii)  a variant of  the Marcinkiewicz
  spectral multiplier theorem  for non-negative self-adjoint operators on product spaces
  when the operators satisfy upper Gaussian heat kernel bounds  (Theorem \ref{th6.1}).

\bigskip

The layout of the paper is as follows.  In Section 2 we recall some basic properties of
heat kernels and finite propagation speed for the wave equation.
  In Section 3 we shall
obtain an atomic decomposition
of  functions for $H^1_{L_1,
L_2}(\mathbb{R}^{n_1}\times\mathbb{R}^{n_2})$  associated with
non-negative self-adjoint operators  with Gaussian  upper bounds on
the heat kernels. The atomic decomposition for elements in the Hardy spaces
$H^1_{L_1,
L_2}(\mathbb{R}^{n_1}\times\mathbb{R}^{n_2})$  (Theorem \ref{theorem of Hardy space atomic decom})
is of independent interest and it
also plays a key role in the proofs of the main result in Section 4, 
namely Theorem \ref{theorem T from Hp to Lp},
which give
 endpoint  estimates  of
boundedness of singular integrals on Hardy spaces 
on product domains.
 Finally, in Section 5, we apply our main results to
deduce  endpoint  estimates  of  the double Riesz transforms  associated to Schr\"odinger operators
with non-negative potentials and  to obtain boundedness  of a Marcinkiewicz-type
  spectral multiplier theorem  for non-negative self-adjoint operators on product spaces.

Throughout, the letter ``$c$"  and ``$C$" will denote (possibly
different) constants  that are independent of the essential
variables.

\vskip 1cm

\section{Backgrounds on heat kernel bounds and spectral multipliers}
\setcounter{equation}{0}

\medskip
This section contains the backgrounds on heat kernel bounds, finite propagation speed of solutions
to the wave equations and spectral multipliers of non-negative self-adjoint operators.

\subsection{Assumptions on heat kernel bounds}\
Unless otherwise specified in the sequel we always assume that $L_i, i=1,2$ are
non-negative self-adjoint operators on $L^2({\mathbb R}^{n_i})$
and that each of
the semigroups $e^{-tL_i}$, generated by $-L_i$ on $L^2({\mathbb R}^{n_i})$,  has the kernel  $p_{t}(x_i,y_i)$
which  satisfies
the following  Gaussian upper bound

$$
|p_{t}(x_i,y_i)| \leq \frac{C}{t^{n_i/2} } \exp\Big(-{|x_i-y_i|^2\over
c\,t}\Big)
\leqno{\rm (GE)}
$$

 \noindent
for all $t>0$,  and $x_i,y_i\in {\mathbb R}^{n_i},$   where $C $ and $ c$   are positive constants.

\medskip

Such estimates are typical for elliptic or sub-elliptic differential operators of second
order (see for instance, \cite{Da} and \cite{DOS}). As the semigroups $e^{-tL_i}$
are holomorphic,  the Gaussian upper bounds for $p_t(x_i,y_i)$ are further
inherited by the time derivatives of $p_{t}(x_i,y_i)$. That is, for each
$k\in{\mathbb N}$, there exist two positive constants $c_k$ and $C_k$ such
that

\begin{eqnarray}\label{time derivative of heat hernel}
\Big|{\partial^k \over\partial t^k} p_{t}(x_i,y_i) \Big|\leq
\frac{C_k}{  t^{n_i /2+k} } \exp\Big(-{|x_i-y_i|^2\over
c_k\,t}\Big)
\end{eqnarray}

\noindent for all $t>0$,  and $x_i,y_i\in {\mathbb R}^{n_i}$. For
the proof of (\ref{time derivative of heat hernel}), see   \cite{Da}
and \cite{Ou}, Theorem~6.17.

\smallskip

\medskip

In what follows, we denote
$$
{\mathbb R}^{n_1+1}_+\times {\mathbb R}^{n_2+1}_+
=\Big\{(x,t):\  x=(x_1, x_2)\in {\mathbb R}^{n_1}\times {\mathbb R}^{n_2}, t=(t_1, t_2),
t_i\geq 0, i=1,2\Big\}.
$$
For any $(x,t)\in {\mathbb R}^{n_1+1}_+\times {\mathbb R}^{n_2+1}_+$ and
 $f\in L^2({\mathbb R}^{n_1}\times {\mathbb R}^{n_2})$,    we set

$$
(e^{-t_1L_1}\otimes e^{-t_2L_2})f(x_1,x_2):=\iint_{\mathbb{R}^{n_1}\times\mathbb{R}^{n_2}}
p_{t_1}(x_1,y_1)p_{t_2}(x_2,y_2)f(y_1,y_2)dy_1dy_2,
$$

\noindent
where we
wish to stress that if $t=0$, then $e^{-tL_i}=1\!\!1_i, i=1,2$, the identity operator on $L^2({\mathbb R}^{n_i})$.

In the absence of regularity on space variables of $p_t(x_i,y_i)$, estimate (\ref{time derivative of heat hernel})
plays an important role in our theory.




\medskip

\subsection{Finite  propagation speed for the wave equation and spectral multipliers}\ Let us
recall that,  if $L$ is a non-negative, self-adjoint operator on $L^2({\mathbb R}^{n})$,
and $E_{L}(\lambda)$ denotes its spectral decomposition, then for every bounded
Borel function $F:[0,\infty)\to{\mathbb{C}}$, one defines the bounded operator
$F(L): L^2({\mathbb R}^{n})\to L^2({\mathbb R}^{n})$ by the formula
\begin{eqnarray}\label{e3.11}
F(L)=\int_0^{\infty}F(\lambda)\,dE_{L}(\lambda).
\end{eqnarray}
In particular, the operator $\cos(t\sqrt{L})$ is then well-defined and bounded
on $L^2({\mathbb R}^{n})$. Moreover, it follows from Theorem 3 of \cite{CS}
  that if  the corresponding
 heat kernels $p_{t}(x,y)$ of $e^{-tL}$ satisfy Gaussian bounds ${\rm (GE)}$, then  there exists a finite,
positive constant $c_0$ with the property that the Schwartz
kernel $K_{\cos(t\sqrt{L})}$ of $\cos(t\sqrt{L})$ satisfies
\begin{eqnarray}\label{e3.12} \hspace{1cm}
{\rm supp}K_{\cos(t\sqrt{L})}\subseteq
\big\{(x,y)\in {\mathbb R}^{n}\times {\mathbb R}^{n}: |x-y|\leq c_0 t\big\}.
\end{eqnarray}
\noindent See also \cite{Si2}. By the Fourier inversion
formula, whenever $F$ is an even, bounded, Borel function with its Fourier transform
$\hat{F}\in L^1(\mathbb{R})$, we can write $F(\sqrt{L})$ in terms of
$\cos(t\sqrt{L})$. More specifically,  we have
\begin{eqnarray}\label{XCP}
F(\sqrt{L})=(2\pi)^{-1}\int_{-\infty}^{\infty}{\hat F}(t)\cos(t\sqrt{L})\,dt,
\end{eqnarray}
which, when combined with (\ref{e3.12}), gives
\begin{eqnarray}\label{e3.13} \hspace{1cm}
K_{F(\sqrt{L})}(x,y)=(2\pi)^{-1}\int_{|t|\geq c_0^{-1}|x-y|}{\hat F}(t)
K_{\cos(t\sqrt{L})}(x,y)\,dt,\qquad \forall\,x,y\in{\mathbb R}^{n}.
\end{eqnarray}

The following  result
(see Lemma~3.5, \cite{HLMMY}) is useful for certain estimates later.

\begin{lemma}\label{lemma finite speed} Let $\varphi\in C^{\infty}_0(\mathbb R)$ be
even, $\mbox{supp}\,\varphi \subset (-c_0^{-1}, c_0^{-1})$, where $c_0$ is
the constant in  \eqref{e3.12}. Let $\Phi$ denote the Fourier transform of
$\varphi$. Then for every $\kappa=0,1,2,\dots$, and for every $t>0$,
the kernel $K_{(t^2L)^{\kappa}\Phi(t\sqrt{L})}(x,y)$ of the operator
$(t^2L)^{\kappa}\Phi(t\sqrt{L})$ which was defined by the spectral theory, satisfies

\begin{eqnarray}\label{e3.11}
{\rm supp}\ \! K_{(t^2L)^{\kappa}\Phi(t\sqrt{L})}(x,y) \subseteq
\Big\{(x,y)\in \mathbb{R}^{n}\times \mathbb{R}^{n}: |x-y|\leq t\Big\}.
\end{eqnarray}
\end{lemma}

Going further, we state the following version. 
(see Lemma~3.5, \cite{HLMMY}).

\begin{lemma}\label{lemma finite speed 2}
  Let
$\varphi\in C^{\infty}_0(\mathbb R)$ be an even function with
$\int\varphi=2\pi$, ${\rm supp}\,\varphi \subset(-1,1)$.
  For every
$m=0,1, 2,\dots$, set $\Phi(\xi):={\hat \varphi}(\xi)$,
$
\Phi^{(m)}(\xi):={d^m\over d\xi^m} \Phi(\xi)
.
$
 Let $\kappa, m\in \mathbb{N}$ and $\kappa + m\in
2\mathbb{N}$. Then   for any $t>0$, the kernel
$K_{(t\sqrt{L})^{\kappa}\Phi^{(m)}(t\sqrt{L})}(x,y)$ of
  $(t\sqrt{L})^{\kappa}\Phi^{(m)}(t\sqrt{L})$ satisfies
\begin{eqnarray}\label{e2.5}\hspace{1.2cm}
{\rm supp}\ \! K_{(t\sqrt{L})^{\kappa}\Phi^{(m)}(t\sqrt{L})}
\subseteq \big\{(x,y)\in \mathbb{R}^{n}\times \mathbb{R}^{n}:\,|x-y|\leq t\big\}
\end{eqnarray}
 and
\begin{eqnarray}\label{ek}
\big|K_{(t\sqrt{L})^{\kappa}\Phi^{(m)}(t\sqrt{L})}(x,y)\big|
\leq C  \, t^{-n}
\end{eqnarray}
for any $x,y\in \mathbb{R}^{n}.$

\end{lemma}

\medskip

Finally, for $s>0$, we define
$$
{\Bbb F}(s)=\Big\{\psi:{\Bbb C}\to{\Bbb C}\ {\rm measurable}: \ \
|\psi(z)|\leq C {|z|^s\over ({1+|z|^{2s}})}\Big\}.
$$
Then for any non-zero function $\psi\in {\Bbb F}(s)$, we have that
$\{\int_0^{\infty}|{\psi}(t)|^2\frac{dt}{t}\}^{1/2}<\infty$.
Denote by  $\psi_t(z)=\psi(tz)$. It follows from the spectral theory
in \cite{Yo} that for any $f\in L^2({\mathbb R}^n)$,

\begin{eqnarray}
\Big\{\int_0^{\infty}\|\psi(t\sqrt{L})f\|_{L^2({\mathbb R}^n)}^2{dt\over t}\Big\}^{1/2}
&=&\Big\{\int_0^{\infty}\big\langle\,\overline{ \psi}(t\sqrt{L})\,
\psi(t\sqrt{L})f, f\big\rangle {dt\over t}\Big\}^{1/2}\nonumber\\
&=&\Big\{\big\langle \int_0^{\infty}|\psi|^2(t\sqrt{L}) {dt\over t}f,
f\big\rangle\Big\}^{1/2}\nonumber\\
&\leq& \kappa \|f\|_{L^2({\mathbb R}^n)}, \label{e2.155}
\end{eqnarray}

\noindent where $\kappa=C_L\big\{\int_0^{\infty}|{\psi}(t)|^2
{dt/t}\big\}^{1/2},$ an estimate which will be often used in the
sequel.
\medskip

\vskip 1cm

\section{Hardy space $H^1_{L_1,
L_2}(\mathbb{R}^{n_1}\times\mathbb{R}^{n_2})$ and its atomic decomposition characterization }
\setcounter{equation}{0}

We shall work exclusively with the domain  $
 {\Bbb R}^{n_{1}+1}_+\times{\Bbb R}^{n_{2}+1}_+$ and its distinguished boundary, ${\Bbb R}^{n_1}\times{\Bbb R}^{n_2}$.
If $x=(x_1,x_2)\in {\Bbb R}^{n_1}\times{\Bbb R}^{n_2}$, $\Gamma(x)$
will denote the product cone  $\Gamma(x)=\Gamma(x_1)\times
\Gamma(x_2)$ where $\Gamma(x_1)=\{(y_1,t_1)\in {\Bbb R}^{n_1+1}_+:
|x_1-y_1|<t_1\} $ and $\Gamma(x_2)=\{(y_2,t_2)\in {\Bbb
R}^{n_2+1}_+: |x_2-y_2|<t_2\}. $ If $(x,t)\in {\Bbb
R}^{n_1+1}_+\times{\Bbb R}^{n_2+1}_+$, then $R_{x,t}$ will denote
the rectangle centered at $x\in {\Bbb R}^{n_1}\times{\Bbb R}^{n_2}$
whose side lengths are $t_1$ and $t_2$. For any open set
$\Omega\subset {\Bbb R}^{n_1}\times{\Bbb R}^{n_2}$, the tent over
$\Omega$, $T(\Omega)$, is the set

\begin{eqnarray} \label{tent over Omega}
\Big{\{}(x,t)\in {\Bbb R}^{n_1+1}_+\times{\Bbb R}^{n_2+1}_+:\
R_{x,t}\subseteq \Omega  \Big{\}}.
\end{eqnarray}

Suppose  that $L_i, i=1,2,$ are  non-negative self-adjoint operator on
$L^2({\mathbb R}^{n_i})$ such that the corresponding
 heat kernels $p_{t_i}(x,y)$ satisfy Gaussian bounds ${\rm (GE)}$.
Given a function $f$ on ${\Bbb R}^{n_1}\times{\Bbb R}^{n_2}$, the area integral function
$Sf$ associated with an operator $L$ is defined by

\begin{eqnarray}\label{esf}
Sf(x)= \bigg(\iint_{\Gamma(x) }\big|\big(t_1^2L_1e^{-t_1^2L_1}\otimes
t_2^2L_2e^{-t_2^2L_2}\big)f(y)\big|^2\
{dy \ \! dt\over t_1^{n_1+1} t_2^{n_2+1}}\bigg)^{1/2}. \label{e3.2}
\end{eqnarray}

\noindent    It is known
that for   $  1<p<\infty,$
 there exist constants $C_1,
C_2$ (which depend on $p$) such that $0<C_1\leq C_2<\infty$ and

\begin{eqnarray}\label{S-function bd on Lp}
C_1\|f\|_p\leq \|Sf\|_p\leq C_2\|f\|_p.
\end{eqnarray}


We adopt the following definition from \cite{DSTY}.

\begin{definition} \label{def of Hardy space via S function}
For $i=1,2$,  let each $L_i$ be a non-negative self-adjoint operator on
$L^2({\mathbb R}^{n_i})$ such that the corresponding
 heat kernels $p_{t_i}(x,y)$ satisfy Gaussian bounds ${\rm (GE)}$.
 The Hardy space $H^1_{{L_1,L_2}}(\mathbb{R}^{n_1}\times\mathbb{R}^{n_2})$
associated to $L_1$ and $L_2$ is defined as the completion of

$$\{ f\in L^2(\mathbb{R}^{n_1}\times\mathbb{R}^{n_2}): \|Sf\|_{L^1(\mathbb{R}^{n_1}\times\mathbb{R}^{n_2})}<\infty \}$$

\noindent with respect to the norm

$$
\|f\|_{H^{1}_{{L_1,L_2}}(\mathbb{R}^{n_1}\times\mathbb{R}^{n_2})
}=\|Sf \|_{L^1(\mathbb{R}^{n_1}\times\mathbb{R}^{n_2})}.
$$
\end{definition}

\bigskip

\noindent
{\bf Remarks.} \

\smallskip
 (i) Note first that   $H^1_{{L_1,L_2}}(\mathbb{R}^{n_1}\times\mathbb{R}^{n_2})$ is a normed linear space. By a standard argument of functional
 analysis (\cite{Yo}) that   $H^1_{{L_1,L_2}}(\mathbb{R}^{n_1}\times\mathbb{R}^{n_2})$ is a Banach space.

\smallskip

(ii) Let  $L_i, i=1,2$ be the Laplacian $\triangle_{n_i}$  on ${\mathbb R}^{n_i}$. It follows from area integral characterization
of Hardy space by using convolution
that the   Hardy  space $H^1(\mathbb{R}^{n_1}\times\mathbb{R}^{n_2})$ coincides with
the spaces   $H^1_{\triangle_{n_1}, \triangle_{n_2}}(\mathbb{R}^{n_1}\times\mathbb{R}^{n_2})$
 and
their norms are equivalent. See \cite{CF1, CF2, F1}.

 \bigskip

\noindent {\bf 3.1. Atomic decomposition for $H^1_{L_1,
L_2}({\mathbb R}^{n_1}\times {\mathbb R}^{n_2})$.}\
 Suppose $\Omega\subset {\Bbb R}^{n_1}\times{\Bbb R}^{n_2}$
is open of finite measure. Denote by $m(\Omega)$ the maximal dyadic
subrectangles of $\Omega$. Let $m_1(\Omega)$  denote those dyadic
subrectangles $R\subseteq \Omega, R=I\times J$ that are maximal in
the $x_1$ direction. In other words if $S=I'\times J\supseteq R$ is
a dyadic subrectangle of $\Omega$, then $I=I'.$ Define $m_2(\Omega)$
similarly.   Let
$${\widetilde \Omega}=\big\{x\in {\mathbb R}^{n_1}\times\mathbb{R}^{n_2}:
M_s(\chi_{\Omega})(x)>{1\over 2}\big\},
$$

\noindent where $M_s$ is the strong maximal operator defined as
$$
   M_s(f)(x)=\sup_{R:\ {\rm \ rectangles\ in\ } {\mathbb R}^{n_1}\times\mathbb{R}^{n_2},\ x\in R}{1\over |R|}\int_R|f(y)|dy.
$$

\noindent For any $R=I\times J\in m_1(\Omega)$, we  set
$\gamma_1(R)=\gamma_1(R, \Omega)=\sup{|l|\over |I|},$ where the
supremum is taken over all dyadic intervals $l: I\subset l$ so that
$l\times J\subset {\widetilde \Omega}$. Define $\gamma_2$ similarly.
Then Journ\'e's lemma, (in one of its forms) says, for  any
$\delta>0$,
\begin{eqnarray*}
\sum_{R\in m_2(\Omega)} |R|\gamma_1^{-\delta}(R)\leq
c_{\delta}|\Omega|
 \ \ \ {\rm and}\ \ \
 \sum_{R\in m_1(\Omega)} |R|\gamma_2^{-\delta}(R)\leq c_{\delta}|\Omega|
 \end{eqnarray*}

\noindent for some  $c_{\delta}$ depending only on $\delta$, not on
$\Omega.$

We now introduce the notion of
$(H^1_{L_1, L_2}, 2, M)$-atom associated to operators.

\begin{definition}\label{def of product atom}  Let $M$ be a positive integer.
A function $a(x_1, x_2)\in L^2({\Bbb R}^{n_1}\times{\Bbb R}^{n_2})$ is called a
$(H^1_{L_1, L_2}, 2, M)$-atom     if it satisfies

\medskip

\noindent $ 1)$  {\rm supp} $a\subset \Omega$, where $\Omega$
is an open set of ${\Bbb R}^{n_1}\times{\Bbb R}^{n_2}$ with finite measure;

\medskip

\noindent $ 2)$ $a$ can be further decomposed into
$$
a=\sum\limits_{R\in m(\Omega)} a_R
$$

\noindent where $m(\Omega)$ is the set of all maximal dyadic
subrectangles of $\Omega$, and there exists a series of function $b_R $ belonging to the range of
$L_1^{k_1}\otimes L_2^{k_2}$ in $L^2({\Bbb R}^{n_1}\times{\Bbb R}^{n_2})$, for each
$k_1, k_2=1, \cdots, M,$   such that

\smallskip

 (i) \ \ \ $a_R=\big(L_1^{M} \otimes L_2^{M}\big) b_R$;

 \smallskip

 (ii) \ \ {\rm supp}\  $\big(L_1^{k_1}  \otimes L_2^{k_2}\big)b_R\subset
10R$, \ \ $ k_1, k_2=0, 1, \cdots, M$;

\smallskip

 (iii) \ $||a||_{L^2({\Bbb R}^{n_1}\times{\Bbb R}^{n_2})}\leq
|\Omega|^{-{1\over 2}}$ and

$$
\sum_{R=I_R\times J_R\in m(\Omega)}\ell(I_R)^{-4M} \ell(J_R)^{-4M}
\Big\|\big(\ell(I_R)^2 \, L_1\big)^{k_1}\otimes \big(\ell(J_R)^2 \,
L_2\big)^{k_2} b_R\Big\|_{L^2({\Bbb R}^{n_1}\times{\Bbb
R}^{n_2})}^2\leq |\Omega|^{-1}.
$$
\end{definition}

\medskip

We are now able to define an atomic $H^1_{L_1,L_2,at,M}$ space,
which we shall eventually show that it is equivalent to the space
$H^1_{L_1,L_2}$ via square functions.

\medskip

\begin{definition}\label{def-of-atomic-product-Hardy-space}
Let $M>\max\{n_1, n_2\}/4$. The Hardy spaces $H^1_{L_1, L_2, at,M}(\mathbb{R}^{n_1}\times\mathbb{R}^{n_2})$
is defined as follows.
  We   say that $f=\sum\lambda_ja_j$ is an atomic
$(H^1_{L_1, L_2}, 2, M)$-representation of $f$ if $\{\lambda_j\}_{j=0}^\infty\in
\ell^1$, each $a_j$ is a $(H^1_{L_1, L_2}, 2, M)$-atom, and the sum converges in
$L^2(\mathbb{R}^{n_1}\times\mathbb{R}^{n_2})$. Set

$$
\mathbb{H}^1_{L_1, L_2, at, M}(\mathbb{R}^{n_1}\times\mathbb{R}^{n_2})=\big\{f: f {\rm\ has\
an\ atomic\ } (H^1_{L_1, L_2}, 2, M)-{\rm representation}\big\},
$$

\noindent with the norm given by

\begin{eqnarray}\label{Hp norm}
&&\|f\|_{\mathbb{H}^1_{L_1, L_2, at, M}(\mathbb{R}^{n_1}\times\mathbb{R}^{n_2})}\\
 &&\
\ =\inf\Big\{ \sum_{j=0}^\infty |\lambda_j|: f=\sum_j\lambda_ja_j
{\rm\ is\ an\ atomic\ } (H^1_{L_1, L_2}, 2, M) {\rm -representation}
\Big\}.\nonumber
\end{eqnarray}

\noindent The space
$H^1_{L_1, L_2, at,M}(\mathbb{R}^{n_1}\times\mathbb{R}^{n_2})$ is then
defined as the completion of
$\mathbb{H}^1_{L_1, L_2,at,M}(\mathbb{R}^{n_1}\times\mathbb{R}^{n_2})$ with
respect to this norm.
\end{definition}

\medskip

We shall say see that any fixed choice of $M>\max\{n_1, n_2\}/4$, the Hardy spaces
$H^1_{L_1, L_2, at,M}(\mathbb{R}^{n_1}\times\mathbb{R}^{n_2})$ yield the same space.
Indeed, we shall show that the ``square function" and ``atom" $H^1$ spaces are
equivalent, if the parameter $M>\max\{n_1, n_2\}/4$ in the next theorem.

\begin{theorem}\label{theorem of Hardy space atomic decom}
Suppose that  $M>\max\{n_1, n_2\}/4$. Then
$$
H^1_{{L_1,L_2}}(\mathbb{R}^{n_1}\times\mathbb{R}^{n_2})=H^1_{L_1,L_2,at,M}(\mathbb{R}^{n_1}\times\mathbb{R}^{n_2}).
$$
Moreover,
$$
\|f\|_{H^1_{{L_1,L_2}}(\mathbb{R}^{n_1}\times\mathbb{R}^{n_2})}\approx
\|f\|_{H^1_{L_1,L_2,at,M}(\mathbb{R}^{n_1}\times\mathbb{R}^{n_2})},
$$

\noindent where the implicit constants depend only on $M$, $n_1$ and
$n_2$.
\end{theorem}

\medskip

\noindent
{\bf 3.2.\ Proof of Theorem~\ref{theorem of Hardy space atomic decom}.}
 We now proceed to the proof of Theorem~\ref{theorem of Hardy space atomic decom}. The basic strategy
 is as follows: by density, it is enough to show that when $
 M>\max({n_1\over 4},{n_2\over 4}),$

 $$\mathbb{H}^1_{L_1,L_2,at,M}(\mathbb{R}^{n_1}\times\mathbb{R}^{n_2})=
 H^1_{L_1,L_2 }(\mathbb{R}^{n_1}\times\mathbb{R}^{n_2})\cap
L^2(\mathbb{R}^{n_1}\times\mathbb{R}^{n_2})
 $$

 \noindent
 with equivalent of norms. The proof of this fact
 proceeds in two steps.

\medskip

 \noindent
 {\bf Step 1.} \ $\mathbb{H}^1_{L_1,L_2,at,M}(\mathbb{R}^{n_1}\times\mathbb{R}^{n_2})\subseteq
 H^1_{L_1,L_2}(\mathbb{R}^{n_1}\times\mathbb{R}^{n_2})\cap
L^2(\mathbb{R}^{n_1}\times\mathbb{R}^{n_2}),$  if
$M>\max({n_1\over 4},{n_2\over 4})$.

\medskip

 \noindent
 {\bf Step 2.} \ $
 H^1_{L_1,L_2}(\mathbb{R}^{n_1}\times\mathbb{R}^{n_2})\cap
L^2(\mathbb{R}^{n_1}\times\mathbb{R}^{n_2})\subseteq \mathbb{H}^1_{L_1,L_2,at,M}(\mathbb{R}^{n_1}\times\mathbb{R}^{n_2}),$  for every
$M\in{\mathbb N}$.

\medskip

We take these in order.
The conclusion of Step 1 is an immediate consequence of the following pair of Lemmas.

\medskip

\begin{lemma}\label{lemma-of-T-bd-on product Hp}
Fix $M\in{\mathbb N}$. Assume that $T$ is a linear
operator, or a non-negative sublinear operator, satisfying the
weak-type (2,2) bound

\begin{eqnarray*}
\big|\ \{x\in \mathbb{R}^{n_1}\times\mathbb{R}^{n_2}:
|Tf(x)|>\eta\}\ \big|\leq
C_T\eta^{-2}\|f\|_{L^2(\mathbb{R}^{n_1}\times\mathbb{R}^{n_2})}^2,\
\ \forall \eta>0,
\end{eqnarray*}

\noindent and that for every $(H^1_{L_1, L_2}, 2, M)$-atom $a$, we have

\begin{eqnarray}\label{e4.111}
\|Ta\|_{L^1(\mathbb{R}^{n_1}\times\mathbb{R}^{n_2})}\leq C
\end{eqnarray}

\noindent with constant $C$ independent of $a$. Then $T$ is bounded
from
$\mathbb{H}^1_{L_1,L_2,at,M}(\mathbb{R}^{n_1}\times\mathbb{R}^{n_2})$
to $L^1(\mathbb{R}^{n_1}\times\mathbb{R}^{n_2})$, and
$$ \|Tf\|_{L^1(\mathbb{R}^{n_1}\times\mathbb{R}^{n_2})}\leq C\|f\|_{\mathbb{H}^1_{L_1,L_2,at,M}(\mathbb{R}^{n_1}\times\mathbb{R}^{n_2})}. $$

\noindent Consequently, by density, $T$ extends to a bounded
operator from
$H^1_{L_1,L_2,at,M}(\mathbb{R}^{n_1}\times\mathbb{R}^{n_2})$ to
$L^1(\mathbb{R}^{n_1}\times\mathbb{R}^{n_2})$.
\end{lemma}

\begin{proof}
Let
$f \in \mathbb{H}^1_{L_1,L_2,at,M}(\mathbb{R}^{n_1}\times\mathbb{R}^{n_2}) $, where $f = \sum \lambda_j a_j$
is an atomic $(H^1_{L_1, L_2}, 2, M)$-representation such that

$$
\|f\|_{ \mathbb{H}^1_{L_1,L_2,at,M}(\mathbb{R}^{n_1}\times\mathbb{R}^{n_2})} \approx  \sum_{j=0}^\infty |\lambda_j| .
$$

\noindent
 Since the sum converges in $L^2$
(by definition), and since $T$ is of weak type $(2,2)$, we have that
at almost every point,

\begin{equation}\label{eq4.44}|T(f)| \leq \sum_{j=0}^\infty |\lambda_j| \,|T(a_j)|.
\end{equation}

\noindent
Indeed, for every $\eta >0$, we have that, if $f^N:= \sum_{j>N} \lambda_j a_j$, then,

 \begin{eqnarray*}
\big|\ \{x: |Tf(x)| - \sum_{j=0}^\infty |\lambda_j| \,|T a_j(x)| >\eta\}\big|\,&\leq& \limsup_{N\to \infty}
\big|  \{x:  |Tf^N(x)|>\eta\}\big|\\
&\leq& \,C_T\,\,\eta^{-2} \,\limsup_{N\to \infty} \|f^N\|_2^2 =0,
\end{eqnarray*}
from which (\ref{eq4.44}) follows.
In turn, (\ref{eq4.44}) and (\ref{e4.111}) imply the desired $L^1$ bound for $Tf$.
\end{proof}

\medskip

\begin{lemma}\label{leAtom}
Let $a$ be an $(H^1_{L_1,L_2}, 2, M)$ atom with
$M>\max({n_1/4},{n_2/4})$.  Let $S$ denote  the square function
defined in \eqref{esf}. Then   for every $(H^1_{L_1, L_2}, 2, M)$-atom
$a$, we have

\begin{eqnarray}\label{e4.11}
\|Sa\|_{1}\leq C,
\end{eqnarray}

\noindent where $C$ is a positive constant independent of $a$.
\end{lemma}

\medskip

\begin{proof} Indeed, given Lemma~\ref{leAtom}, we may apply Lemma~\ref{lemma-of-T-bd-on product Hp} with $T=S$ to obtain

$$
\|f\|_{H^1_{L_1,
L_2}(\mathbb{R}^{n_1}\times\mathbb{R}^{n_2})}=\|Sf\|_{L^1(\mathbb{R}^{n_1}\times\mathbb{R}^{n_2})}\leq
C\|f\|_{\mathbb{H}^1_{L_1,L_2,at,M}(\mathbb{R}^{n_1}\times\mathbb{R}^{n_2})},
$$

\noindent
which Step 1 follows.

To finish Step 1, it therefore suffices to   verify   estimate (\ref{e4.11}). To see this, we need to apply the
Journ\'e's covering lemma. For any $(H_{L_1,L_2}^1,2,M)$-atom $a$, suppose that
 $
a=\sum\limits_{R\in m(\Omega)} a_R
$  is supported in an open set $\Omega$ with finite measure.
For any $R=I \times J\in m(\Omega)$, let $\widetilde{I}$ be the
biggest dyadic cube containing  $I$, so that $\widetilde{I}\times
J\subset\widetilde{\Omega}$, where
$\widetilde{\Omega}=\{x\in\mathbb{R}^{n_1}\times\mathbb{R}^{n_2}:\
M_s(\chi_{\Omega})(x)>1/2\}$. Next, let $\widetilde{J}$ be the
biggest dyadic cube containing $J$, so that $\widetilde{I}\times
\widetilde{J}\subset\widetilde{\widetilde{\Omega}}$, where
$\widetilde{\widetilde{\Omega}}=\{x\in\mathbb{R}^{n_1}\times\mathbb{R}^{n_2}:\
M_s(\chi_{\widetilde{\Omega}})(x)>1/2\}$. Now let $\widetilde{R}$ be
the 100-fold dilate of $\widetilde{I}\times \widetilde{J}$ concentric with
$\widetilde{I}\times \widetilde{J}$. Clearly, an application of the
strong maximal function theorem shows that
$\big|\cup_{R\subset\Omega} \widetilde{R}\big|\leq
C|\widetilde{\widetilde{\Omega}}|\leq C|\widetilde{\Omega}|\leq
C|\Omega|$. From the property (iii) of the
$(H_{L_1,L_2,}^1,2,M)$-atom,

\begin{eqnarray}\label{SL alpha uniformly bd on inside of Omega   }
\int_{\cup \widetilde{R}}|S(a)(x_1,x_2)|dx_1dx_2 &\leq&
 |\cup\widetilde{R} |^{1/2}\|S(a)\|_{L^2(\mathbb{R}^{n_1}\times\mathbb{R}^{n_2})}\nonumber\\[2pt]
 &\leq&  C |\Omega|^{1/2}\|a\|_{L^2(\mathbb{R}^{n_1}\times\mathbb{R}^{n_2})}\\[2pt]
 &\leq&    C |\Omega|^{1/2}|\Omega|^{-1/2}
 \leq  C.\nonumber
\end{eqnarray}

\noindent
We now prove

\begin{eqnarray}\label{SL alpha uniformly bd on outside of Omega   }
\int_{(\bigcup \widetilde{R})^c}|S(a)(x_1,x_2)|dx_1dx_2\leq C.
\end{eqnarray}

\noindent
From the definition of $a$, we write

\begin{eqnarray}\label{DE}
\int_{(\bigcup \widetilde{R})^c}|S(a)(x_1,x_2)|dx_1dx_2 &\leq&
\sum_{R\in
m(\Omega) } \int_{\widetilde{R}^c}|S(a_R)(x_1,x_2)|dx_1dx_2\nonumber\\[3pt]
&\leq &\sum_{R\in m(\Omega) }
 \int_{(100\widetilde{I})^c\times\mathbb{R}^{n_2}} |S(a_R)(x_1,x_2)|dx_1dx_2\nonumber\\[3pt]
&&+ \sum_{R\in m(\Omega) }
 \int_{\mathbb{R}^{n_1}\times
(100\widetilde{J})^c}  |S(a_R)(x_1,x_2)|dx_1dx_2\nonumber\\[3pt]
&=& D+E.
\end{eqnarray}

\noindent
For the term $D$, we have
\begin{eqnarray*}\label{D1D2}
\int_{(100\widetilde{I})^c\times\mathbb{R}^{n_2}}|S(a_R)(x_1,x_2)|dx_1dx_2
 &  =&
\int_{(100\widetilde{I})^c\times 100J} |S(a_R)(x_1,x_2)|dx_1dx_2 \nonumber\\[2pt]
&&+  \int_{(100\widetilde{I})^c\times (100J)^c} |S(a_R)(x_1,x_2)|dx_1dx_2 \nonumber\\[2pt]
&=& D_1+D_2.
\end{eqnarray*}

\noindent Let us first estimate the term $D_1$.  Set $a_{R,2}=(1\!\!1_1\otimes L_2^M)b_R$, that is,
$
a_R= (L_1^M\otimes 1\!\!1_2)a_{R,2}.
$    Using H\"older's
inequality,

\begin{eqnarray}\label{estimate D1} \ \ \
D_1
&\leq& C |J|^{{1/2}} \int_{(100\widetilde{I})^c }\Big( \int_{100J}|S(a_R)(x_1,x_2)|^2dx_2\Big)^{1/2}dx_1.
\end{eqnarray}

\noindent
We will show that there exists a constant $C>0$ such that for any $x_1\not\in 100\widetilde{I},$

\begin{eqnarray}\label{estimate D11} \ \ \
 &&\hspace{-1.5cm}\int_{100J}|S(a_R)(x_1,x_2)|^2dx_2\nonumber\\
 &\leq& C{|I|^{1/n_1+1} \over |x_1-x_I|^{2n_1+1}} \ell(I)^{-4M}\ell(J)^{-4M}
\|(1\!\!1_1\otimes
(\ell(J)^2L_2)^M)b_{R}\|^2_{L^2(\mathbb{R}^{n_1}\times\mathbb{R}^{n_2})},
\end{eqnarray}

\noindent
where $(x_I, x_J)$ denotes the center of  $R=I\times J$.

\noindent

Let us verify (\ref{estimate D11}). From the definition of the $S$-function and the $L^2$-boundedness of
the area function of the  one-parameter,

\begin{eqnarray}\label{ed222}\hspace{1cm}
&&\hspace{-1.2cm}\int_{100J}|S(a_R)(x_1,x_2)|^2dx_2\nonumber\\[4pt]
 &\leq& \int_{\Gamma_1(x_1)}\bigg[
\int_{\mathbb{R}^{n_2}}\int_{\Gamma_2(x_2)}\big|t_2^2L_2e^{-t_2^2L_2}\big(t_1^2L_1e^{-t_1^2L_1}
a_{R}(y_1,\cdot)\big)(y_2)\big|^2{dy_2dt_2\over
t_2^{n_2+1}}dx_2 \bigg]{dy_1dt_1\over t_1^{n_1+1}}\nonumber\\[4pt]
 &\leq& C\int_{\Gamma_1(x_1)}
\int_{\mathbb{R}^{n_2}}\big|t_1^2L_1e^{-t_1^2L_1}
a_{R}(y_1,x_2)\big|^2dx_2{dy_1dt_1\over t_1^{n_1+1}}\nonumber\\[4pt]
&\leq& C
\int_{\mathbb{R}^{n_2}}\int_{\Gamma_1(x_1)} \Big|(t_1^2L_1)^{M+1}e^{-t_1^2L_1}
a_{R,2}(y_1,x_2)\Big|^2{dy_1dt_1\over t_1^{n_1+1+4M}}dx_2,
\end{eqnarray}

\noindent
where the last inequality follows from   the equality  $
a_R= (L_1^M\otimes 1\!\!1_2)a_{R,2}$.  Note that  supp$a_{R,2}\subset$ supp$ (1\!\!1_1\otimes L_2^M)b_R\subset
10R=10(I\times J)$. We then apply the time derivatives  (\ref{time derivative of heat hernel})  of the kernel
$p_t(x_i, y_i)$ to obtain

 \begin{eqnarray*}
 &&\hspace{-1.2cm}\mbox{RHS of (\ref{ed222})} \\[2pt]
 &\leq& C  \int_{\mathbb{R}^{n_2}}
 \int_0^{\ell(I)} \int_{|x_1-y_1|<t_1} \Big[\int_{10I}  t_1^{-n_1} \exp\Big(-{|y_1-z_1|^2\over ct_1^2}\Big)
|a_{R,2}(z_1,x_2)|dz_1\Big]^2 {dy_1dt_1\over t_1^{n_1+1+4M}} dx_2\nonumber \\[4pt]
&&+ C  \int_{\mathbb{R}^{n_2}}
 \int_{\ell(I)}^\infty
 \int_{|x_1-y_1|<t_1} \Big[\int_{10I}  t_1^{-n_1} \exp\Big(-{|y_1-z_1|^2\over ct_1^2}\Big)
|a_{R,2}(z_1,x_2)|dz_1\Big]^2 {dy_1dt_1\over t_1^{n_1+1+4M}} dx_2\nonumber \\[4pt]
 &=:&   D_{1} (a_R)(x_1) +D_{2} (a_R)(x_1).\nonumber
\end{eqnarray*}

\noindent




Let us estimate the term $D_{1} (a_R)(x_1)$. Note that if $x_1\not\in 100\widetilde{I},
$ $0<t_1<\ell(I)$, $|x_1-y_1|<t_1$ and $z_1\in 10I$, then $|y_1-z_1|\geq |x_1-x_I|/2$.
 We use the fact that $e^{-s}\leq Cs^{-k}$  for any $k>0$ to obtain

\begin{eqnarray*}
  D_{1} (a_R)(x_1)&\leq& C
 \int_0^{\ell(I)} \int_{|x_1-y_1|<t_1}dy_1\ \ t_1^{-2n_1} \exp\Big(-{2|x_1-x_I|^2\over ct_1^2}\Big)  {dt_1\over t_1^{n_1+1+4M}}\\
 &&\times \int_{\mathbb{R}^{n_2}}\Big[\int_{10I}
|a_{R,2}(z_1,x_2)|dz_1\Big]^2  dx_2\\
  &\leq&
 C|I|\int_0^{\ell(I)} t_1^{-2n_1} \exp\Big(-{2|x_1-x_I|^2\over ct_1^2}\Big) {dt_1\over t_1^{1+4M}} \|a_{R,2}\|^2_{L^2(\mathbb{R}^{n_1}\times\mathbb{R}^{n_2})}\\
&\leq& C|I|
 \int_0^{\ell(I)} t_1^{-2n_1-4M-1} \Big({ t_1\over |x_1-x_I|}\Big)^{2(n_1+2M +{1\over 2} )}
   {dt_1 } \|a_{R,2}\|^2_{L^2(\mathbb{R}^{n_1}\times\mathbb{R}^{n_2})}\\
   &\leq& C|I|
 {\ell(I) \over |x_1-x_I|^{2(n_1+2M +{1\over 2})}}
   \|(1\!\!1_1\otimes
 L_2^M)b_{R}\|^2_{L^2(\mathbb{R}^{n_1}\times\mathbb{R}^{n_2})}\\
&\leq& C  {|I|^{1/n_1 +1}  \over |x_1-x_I|^{2n_1+1}} \ell(I)^{-4M}\ell(J)^{-4M}
\|(1\!\!1_1\otimes
(\ell(J)^2L_2)^M)b_{R}\|^2_{L^2(\mathbb{R}^{n_1}\times\mathbb{R}^{n_2})},
\end{eqnarray*}

\noindent
which is of the right order. In order to estimate the second term $D_{2} (a_R)$, observe that
  if $x_1\not\in 100\widetilde{I},
$ $\ell(I)\leq t_1 <|x_1-x_I|/4$, $|x_1-y_1|<t_1$ and $z_1\in 10I$, then $|y_1-z_1|\geq  |x_1-x_I|/4$.
Hence,

\begin{eqnarray*}
 &&\hspace{-0.6cm}D_{2} (a_R)(x_1) \nonumber \\[2pt]
 &\leq&
C\int_{\mathbb{R}^{n_2}}
 \int_{\ell(I)}^{|x_1-x_I|\over 4} \int_{|x_1-y_1|<t_1} \Big[\int_{10I}  t_1^{-n_1} \exp\Big(-{|y_1-z_1|^2\over ct_1^2}\Big)
|a_{R,2}(z_1,x_2)|dz_1\Big]^2 {dy_1dt_1\over t_1^{n_1+1+4M}} dx_2\nonumber \\[2pt]
&&+
C\int_{\mathbb{R}^{n_2}}
 \int_{|x_1-x_I|\over 4}^\infty
 \int_{|x_1-y_1|<t_1} \Big[\int_{10I}  t_1^{-n_1} \exp\Big(-{|y_1-z_1|^2\over ct_1^2}\Big)
|a_{R,2}(z_1,x_2)|dz_1\Big]^2 {dy_1dt_1\over t_1^{n_1+1+4M}} dx_2\nonumber \\[2pt]
 &\leq&  C|I| \bigg(
 \int_{\ell(I)}^{\infty} t_1^{-2n_1-1-4M} \Big({ t_1\over |x_1-x_I|}\Big)^{2(n_1+2M -{1\over 2})} {dt_1 }
 +  \int_{|x_1-x_I|\over 4}^{\infty} t_1^{-2n_1-1-4M}    dt_1 \bigg)
 \|a_{R,2}\|^2_{L^2(\mathbb{R}^{n_1}\times\mathbb{R}^{n_2})}\\[2pt]
&\leq& C  {|I|^{1/n_1+1} \over |x_1-x_I|^{2n_1+1}} \ell(I)^{-4M}\ell(J)^{-4M}
\|(1\!\!1_1\otimes
(\ell(J)^2L_2)^M)b_{R}\|^2_{L^2(\mathbb{R}^{n_1}\times\mathbb{R}^{n_2})}.
\end{eqnarray*}

\noindent

\noindent
Combining the estimates of $D_{1} (a_R)(x_1)$ and  $D_{2} (a_R)(x_1)$, estimate (\ref{estimate D11})
 follows readily. Putting (\ref{estimate D11}) into the term $D_1$ in (\ref{estimate D1}),
  we    have

\begin{eqnarray*}
 D_1 &\leq& C|R|^{1/2} \int_{(100\widetilde{I})^c } {|I|^{1/2n_1 } \over |x_1-x_I|^{n_1+1/2}}  dx_1  \ell(I)^{-2M}\ell(J)^{-2M}
\|(1\!\!1_1\otimes
(\ell(J)^2L_2)^M)b_{R}\|_{L^2(\mathbb{R}^{n_1}\times\mathbb{R}^{n_2})}\\ [3pt]
&\leq&C|R|^{1/2} \gamma_1(R)^{-1/2}  \ell(I)^{-2M}\ell(J)^{-2M}
\|(1\!\!1_1\otimes
(\ell(J)^2L_2)^M)b_{R}\|_{L^2(\mathbb{R}^{n_1}\times\mathbb{R}^{n_2})}.
\end{eqnarray*}

Now we turn to estimate the term $D_2$. Note that  $
a_R=(L_1^M\otimes L_2^M)b_R
$ and supp $b_R\subset 10R=10(I\times J)$.   One can write

\begin{eqnarray*}
&&\hspace{-0.6cm}\big(S(a_R)(x_1,x_2)\big)^2\\[5pt]
&&\
=\int\limits_{\Gamma_1(x_1)}\int\limits_{\Gamma_2(x_2)}\Big|\big((t_1^2L_1)^{M+1}e^{-t_1^2L_1}\otimes
(t_2^2L_2)^{M+1}e^{-t_2^2L_2}\big) b_R (y_1,y_2)\Big|^2 {dy_1dt_1\over
t_1^{n_1+4M+1}}{dy_2dt_2\over t_2^{n_2+4M+1}}\\[2pt]
&&\ \leq C\bigg( \int_0^{\ell(I)}\!\!\int_0^{\ell(J)} +
\int_0^{\ell(I)}\!\!\int_{\ell(J)}^{\infty}+
\int_{\ell(I)}^{\infty}\int_0^{\ell(J)} +\int_{\ell(I)}^{\infty}\int_{\ell(J)}^{\infty} \Big)
\int_{|x_1-y_1|<t_1} \int_{|x_2-y_2|<t_2}\bigg[\int_{10I} \int_{10J}\\[2pt]
&& \hspace{0.5cm}  t_1^{-n_1} \exp\Big(-{|y_1-z_1|^2\over ct_1^2}\Big)
   t_2^{-n_2} \exp\Big(-{|y_2-z_2|^2\over ct_2^2}\Big)
|b_{R}(z_1,z_2)|dz_1 dz_2 \bigg]^2
{dy_1dt_1\over
t_1^{n_1+4M+1}}{dy_2dt_2\over t_2^{n_2+4M+1}}\\[2pt]
&&\
=: \sum_{i=1}^4 D_{2i}(b_R)(x_1, x_2).
\end{eqnarray*}

\noindent  By using an argument as in $D_{1}(a_R)$ and $D_{2}(a_R)$ above, together with H\"older inequality and
elementary integration, we can show  that  for every $i=1,2,3, 4,$

\begin{eqnarray*}
D_{2i}(b_R)(x_1, x_2)&\leq& C |R|  {|I|^{1/n_1} \over |x_1-x_I|^{2n_1+1}} \times {|J|^{1/n_2} \over |x_2-x_J|^{2n_2+1}}
  \ell(I)^{-4M}\ell(J)^{-4M} \\[3pt]
&& \hspace{1cm}\times \big\|((\ell(I)^2L_1)^M\otimes
(\ell(J)^2L_2)^M)b_{R}\big\|^2_{L^2(\mathbb{R}^{n_1}\times\mathbb{R}^{n_2})},
\end{eqnarray*}

\noindent which gives
\begin{eqnarray*}
D_2 &\leq&  \int_{(100\widetilde{I})^c\times (100J)^c} |S(a_R)(x_1,x_2)|dx_1dx_2 \nonumber\\[4pt]
&\leq& C |R|^{1/2} \gamma_1(R)^{-1/2}
  \ell(I)^{-2M}\ell(J)^{-2M}   \|((\ell(I)^2L_1)^M\otimes
(\ell(J)^2L_2)^M)b_{R}\|_{L^2(\mathbb{R}^{n_1}\times\mathbb{R}^{n_2})}.
\end{eqnarray*}

\noindent
Estimates of $D_1$ and $D_2$, together with
  H\"older's inequality and Journ\'e's covering lemma, shows that

\begin{eqnarray*}
D&\leq& \sum_{R\in m(\Omega) }
 \int_{(100\widetilde{I})^c\times\mathbb{R}^{n_2}} |S(a_R)(x_1,x_2)|dx_1dx_2\nonumber\\[3pt]
&\leq& \sum_{R\in m(\Omega) }|R|^{1/2} \gamma_1(R)^{-1/2}
  \ell(I)^{-2M}\ell(J)^{-2M}
 \\
&&\hspace{0.4cm} \times\Big( \|((\ell(I)^2L_1)^M\otimes
(\ell(J)^2L_2)^M)b_{R}\|_{L^2(\mathbb{R}^{n_1}\times\mathbb{R}^{n_2})}
+  \|(1\!\!1_1\otimes
(\ell(J)^2L_2)^M)b_{R}\|_{L^2(\mathbb{R}^{n_1}\times\mathbb{R}^{n_2})}\Big)\\[3pt]
&\leq& C\Big(\sum_{R\in m(\Omega) }|R|  \gamma_1(R)^{-1} \Big)^{1/2}
  \Big(\sum_{R\in m(\Omega) } \ell(I)^{-4M}\ell(J)^{-4M}
 \\
&&\hspace{0.4cm} \times\Big( \|((\ell(I)^2L_1)^M\otimes
(\ell(J)^2L_2)^M)b_{R}\|^2_{L^2(\mathbb{R}^{n_1}\times\mathbb{R}^{n_2})}
+  \|(1\!\!1_1\otimes
(\ell(J)^2L_2)^M)b_{R}\|^2_{L^2(\mathbb{R}^{n_1}\times\mathbb{R}^{n_2})}\Big)\Big)^{1/2}\\[3pt]
&\leq& C |\Omega|^{{1\over 2}} |\Omega|^{-{1\over 2}}
 \leq  C.
\end{eqnarray*}

 Similarly, we have that $E\leq C,$ and then the desired estimate of
(\ref{SL alpha uniformly bd on outside of Omega   }) follows readily. This,
together with (\ref{SL alpha uniformly bd on inside of Omega   }),
yields (\ref{e4.11}). This concludes Step 1.
\end{proof}

\medskip

 We now turn to Step 2.
Our goal is to show that every
 $f\in H_{L_1,L_2}^1(\mathbb{R}^{n_1}\times\mathbb{R}^{n_2})\cap
L^2(\mathbb{R}^{n_1}\times\mathbb{R}^{n_2})$ has a $(H^1_{L_1, L_2},
2, M)$ atom representation, with appropriate quantitative control of
the coefficients.  To this end, we follow the standard tent space
approach.

  Let us   recall some basic facts from \cite{AM, GM} on product
domains. First, for $1\leq p<\infty$, the tent spaces on ${\Bbb R}^{n_1}\times  {\Bbb R}^{n_2}$ are defined by

\begin{eqnarray*}
T^{p,2}({\Bbb
R}^{n_1}\times  {\Bbb R}^{n_2}):=\Big\{ F: {\Bbb R}^{n_1+1}_+\times  {\Bbb R}^{n_2+1}_+\rightarrow {\mathbb C};\
\|F\|_{T^{p,2}({\Bbb R}^{n_1}\times  {\Bbb R}^{n_2})}:
=\|\mathcal{A}(F)\|_{L^p({\Bbb R}^{n_1}\times  {\Bbb R}^{n_2})}
<\infty \Big\},
\end{eqnarray*}

\noindent
where

$$
\mathcal{A}F(x)=\bigg(\iint_{\Gamma(x)}|F(y,t)|^2\frac{dydt}
{t_1^{n_1+1}t_2^{n_2+1}}\bigg)^{1/2}.
$$

 \noindent
The tent space $T^{p,2}({\Bbb R}^{n_1}\times  {\Bbb R}^{n_2})$ is
defined as the space of functions $F$ such that ${\mathcal A}(F)\in
L^p({\Bbb R}^{n_1}\times  {\Bbb R}^{n_2})$ when $0<p<\infty$. The
resulting equivalences classes are then equipped with the norm,
$\|F\|_{T^{p,2}({\Bbb R}^{n_1}\times  {\Bbb
R}^{n_2})}=\|\mathcal{A}(F)\|_{L^p({\Bbb R}^{n_1}\times  {\Bbb
R}^{n_2})}.$

It has been proved in \cite{F1} and \cite{DSTY} that every $F\in
T^{1,2}({\mathbb R}\times\mathbb{R})$ has an atomic decomposition. It
is easy to generalize to the case
$T^{p,2}(\mathbb{R}^{n_1}\times\mathbb{R}^{n_2})$. For further
reference, we record this result below.

\begin{definition}\label{deftent}
A function $a(x,t)$ is called a $T^{1,2}$-atom, if there exists an
open set $\Omega\subset {\mathbb R}^{n_1}\times\mathbb{R}^{n_2}$ of
finite measure satisfying the following properties:

\medskip

\noindent   {\rm i)}  \ $a(x,t)$ can be further decomposed as $a=\sum_{R\in
m({\Omega})}a_R$, where each $a_R$ is supported in $T(3R)$, and
$R\subset {\Omega}$ (say, $R=I\times J$ in the sum)  is a maximal
dyadic subrectangle of $\Omega;$

\medskip
\noindent
 {\rm ii)}  $ \|a\|_{L^2(dydt/(t_1t_2))}\leq |\Omega|^{-{1\over
2}} \ \ {\rm and}\ \ \sum_{R\in
m({\Omega})}\|a_R\|^2_{L^2(dydt/(t_1t_2))}\leq |{\Omega}|^{-1}. $

\end{definition}

\medskip

It turns out, as in the one parameter case, we have the following proposition.

\medskip

\begin{prop}\label{prop-of-product tent space}
For every element $F\in
T^{1,2}(\mathbb{R}^{n_1}\times\mathbb{R}^{n_2})$, there exist a
numerical sequence $\{\lambda_j\}_{j=0}^\infty\in\ell^1$ and a
sequence of $T^{1,2}$-atoms $\{A_j\}_{j=0}^\infty$ such that

\begin{eqnarray}\label{atom decomposition for product tent space}
F=\sum_j\lambda_jA_j\ \ \ {\rm in}\
T^{1,2}(\mathbb{R}^{n_1}\times\mathbb{R}^{n_2}) \ {\rm and}\ a.e.\
{\rm in}\ \mathbb{R}^{n_1+1}_+\times\mathbb{R}^{n_2+1}_+.
\end{eqnarray}

\noindent Moreover,
$$ \sum_{j=0}^\infty|\lambda_j|\thickapprox \|F\|_{T^{1,2}(\mathbb{R}^{n_1}\times\mathbb{R}^{n_2})}, $$

\noindent where the implicit constants depend only on $n_1$ and
$n_2$.

Finally, if $F\in
T^{1,2}(\mathbb{R}^{n_1}\times\mathbb{R}^{n_2})\cap
T^{2,2}(\mathbb{R}^{n_1}\times\mathbb{R}^{n_2})$, then the
decomposition (\ref{atom decomposition for product tent space}) also
converges in $T^{2,2}(\mathbb{R}^{n_1}\times\mathbb{R}^{n_2})$.
\end{prop}

\begin{proof}
  Except for the final part of the proposition, concerning $T^{2,2}$
convergence, the results are contained in pp. 841-842, \cite{F1},
also Proposition 3.3 in \cite{DSTY}. And we refer the reader to
those papers for the proof.  To this end, from the definition of
$T^{2,2}$, we have

\begin{eqnarray}\label{product bound of T22 norm of F}
\|F\|_{T^{2,2}(\mathbb{R}^{n_1}\times\mathbb{R}^{n_2})}^2=\int_{\mathbb{R}^{n_1}\times\mathbb{R}^{n_2}}
\mathcal{A}F^2dx\leq C
\int_0^\infty\int_0^\infty\int_{\mathbb{R}^{n_1}\times\mathbb{R}^{n_2}}
|F(y,t)|^2{dydt\over t_1t_2}.
\end{eqnarray}

\noindent Suppose now that $F\in
T^{1,2}(\mathbb{R}^{n_1}\times\mathbb{R}^{n_2})\cap
T^{2,2}(\mathbb{R}^{n_1}\times\mathbb{R}^{n_2})$. We recall that in
the constructive proof of the decomposition (\ref{atom decomposition
for product tent space}) in \cite{DSTY}, one has that

$$
 \lambda_jA_j= F 1_{S_j}, $$

\noindent where $\{S_j\}$ is a collection of the open sets which are
pairwise disjoint (up to sets of measure zero), and whose union
covers $\mathbb{R}^{n_1+1}_+\times\mathbb{R}^{n_2+1}_+$. Thus, by
(\ref{product bound of T22 norm of F}),

\begin{eqnarray*}
\|\sum_{j>N}\lambda_jA_j\|_{T^{2,2}(\mathbb{R}^{n_1}\times\mathbb{R}^{n_2})}^2&\leq&
C \int_0^\infty\int_0^\infty\int_{\mathbb{R}^{n_1}\times\mathbb{R}^{n_2}}
\big|\sum_{j>N} 1_{S_j} F(y,t)\big|^2{dydt\over t_1t_2}\\
&=& \sum_{j>N}\iint_{S_j}\big| F(y,t)\big|^2{dydt\over t_1t_2}
\rightarrow 0
\end{eqnarray*}

\noindent as $N\rightarrow\infty$, where we have used the
disjointness of the sets $S_j$ and dominated convergence. It follows
that $F=\sum_j\lambda_jA_j$ in
$T^{2,2}(\mathbb{R}^{n_1}\times\mathbb{R}^{n_2})$.
\end{proof}

\medskip

Now, given  $M\geq1$, we define an operator $\pi_{L_1,L_2,M}$,
 acting initially on $T^{2,2}(\mathbb{R}^{n_1}\times\mathbb{R}^{n_2}),$
as follows:

\begin{eqnarray}\label{product operator pi LM}
\pi_{L_1,L_2,M}(F)(x)= \int_0^\infty\int_0^\infty
\psi(t_1\sqrt{L_1})\psi(t_2\sqrt{L_2})(F(\cdot,t))(x){dt_1\over
t_1}{dt_2\over t_2},
\end{eqnarray}

\noindent where $\psi(x)=x^{2M}\varphi(x)$ and $\varphi(x)$ is the function mentioned in Lemma \ref{lemma finite speed}.
In particular, $\pi_{L_1,L_2,1}$ will denote by $\pi_{L_1,L_2}.$
By a standard duality argument involving well known quadratic estimates for $L_i, i=1,2$, one obtains that
the improper integral converges weakly in $L^2$, and that for every $M\geq 0,$

\begin{eqnarray}\label{product L2 bound of pi LM}
\|\pi_{L_1,L_2,M}(F)\|_{L^2(\mathbb{R}^{n_1}\times\mathbb{R}^{n_2})}\leq
C_M \|F\|_{T^{2,2}(\mathbb{R}^{n_1}\times\mathbb{R}^{n_2})}.
\end{eqnarray}

\noindent Following \cite{HLMMY}, we now observe that $\pi_{L_1,L_2,M}$
essentially maps $T^{1,2}$ atoms into $H^1_{L_1, L_2}$ atoms. We
have

\begin{lemma}\label{lemma-of-product operator-pi L M}
Suppose that $A$ is a
$T^{1,2}(\mathbb{R}^{n_1}\times\mathbb{R}^{n_2})$-atom associated
with an open set $\Omega\subset
\mathbb{R}^{n_1}\times\mathbb{R}^{n_2}$ with finite measure (or more
precisely, to its tent $T(\Omega)$). Then for every $M\geq 1$, there
is a uniform constant $C_M$ such that $C_M^{-1}\pi_{L_1,L_2,M}(A)$
is a $(H^1_{L_1, L_2}, 2, M)$-atom associated with $\Omega$.
\end{lemma}

\begin{proof}
Fix an open set
$\Omega\subset\mathbb{R}^{n_1}\times\mathbb{R}^{n_2}$ with finite
measure and let $A$ be a $T^{1,2}$-atom satisfying (i) and (ii) in
Definition~\ref{deftent}, that is, $A(x,t)$ can be further
decomposed as $A(x,t)=\sum_{R\in m({\Omega})}A_R(x,t)$, where each
$A_R$ is supported in $T(3R)$, and $R\in m(\Omega)$ is a maximal
dyadic subrectangle of $\Omega$ satisfying $
\|A\|^2_{L^2(dxdt/(t_1t_2))}\leq |{\Omega}|^{-1} $ and $
\sum\limits_{R\in m({\Omega})}\|A_R\|^2_{L^2(dxdt/(t_1t_2))}\leq
|{\Omega}|^{-1}. $ Set

$$
a=\pi_{L_1,L_2,M}(A)=\sum_{R\in m({\Omega})} a_R
$$

\noindent
 where
$
a_R: = \big(L_1^M\otimes L_2^M\big) b_R
$
and

 $$
 b_R= \int_0^\infty\int_0^\infty
t_1^{2M}\varphi(t_1\sqrt{L_1})t_2^{2M}\varphi(t_2\sqrt{L_2})\big(A_R(\cdot,t)\big){dt_1\over
t_1}{dt_2\over t_2},
$$

\noindent where $\varphi$ is the function mentioned in Lemma
\ref{lemma finite speed}. Then it follows from Lemma \ref{lemma
finite speed} that the integral kernel
$K_{(t_i^2L_i)^{k}\Phi(t_i\sqrt  L_i)}(x,y)$ of the operator
$(t_i^2L_i)^{k}\Phi(t_i\sqrt L_i)$ satisfies $ {\rm
supp}\,K_{(t_i^2L_i)^{k}\Phi(t_i\sqrt  L_i)}\subset
\big\{(x_i,y_i)\in{\Bbb R}^{n_i}\times {\Bbb R}^{n_i}:
|x_i-y_i|<t_i\big\} $
 for $i=1,2.$\ \ This, together with  the fact that ${\rm supp}\, A_R\subset T(3R)$, shows
that for every $k_1,k_2=0,1,\dots,M$,
\begin{eqnarray}\label{eq1005}
{\rm supp}\,\big(L_1^{k_1} \otimes L_2^{k_2}\big)b_{R}\subseteq 10R.
\end{eqnarray}

\noindent

Next we estimate $\|a\|_{L^2(\mathbb{R}^{n_1}\times\mathbb{R}^{n_2})}$.
Taking $g\in L^2(\mathbb{R}^{n_1}\times\mathbb{R}^{n_2})$
such that  $\|g\|_{L^2(\mathbb{R}^{n_1}\times\mathbb{R}^{n_2})}=1$, we then use the fact
that    $a=\pi_{L_1,L_2,M}(A)$ to obtain

\begin{eqnarray*}
&&\hspace{-1.4cm}
 \Big|\int_{\mathbb{R}^{n_1}\times\mathbb{R}^{n_2}}
\pi_{L_1,L_2,M}(A)(x)g(x)dx \Big|\\
&=&\Big|\lim\limits_{\delta\rightarrow 0}\int_{\mathbb{R}^{n_1}\times\mathbb{R}^{n_2}}
\bigg(\int_{\delta}^{1/\delta}\int_{\delta}^{1/\delta}
\psi(t_1\sqrt{L_1})\psi(t_2\sqrt{L_2})(A(\cdot,t))(x){dt \over
t_1  t_2}\bigg) g(x)dx \Big|\\
&\leq& \int_{\mathbb{R}^{n_1}\times\mathbb{R}^{n_2}}\int_0^\infty\int_0^\infty
\Big|A(x,t) \psi(t_1\sqrt{L_1})\psi(t_2\sqrt{L_2})\big(g\big)(x)\Big| {dx dt\over t_1 t_2}  \\[2pt]
&\leq & C
\|A\|_{L^2(dxdt/(t_1t_2))}\bigg(\int_0^\infty\int_0^\infty
\int_{\mathbb{R}^{n_1}\times\mathbb{R}^{n_2}}|\psi(t_1\sqrt{L_1})\psi(t_2\sqrt{L_2})\big(g\big)(x)|^2
 {dx dt\over t_1 t_2}\bigg)^{1/2}\\
&\leq & C
|\Omega|^{-{1\over 2}}\|g\|_{L^2(\mathbb{R}^{n_1}\times\mathbb{R}^{n_2})}
 \leq   C |\Omega|^{-{1\over 2}},
\end{eqnarray*}

\noindent and we have  $
\|a\|_{L^2(\mathbb{R}^{n_1}\times\mathbb{R}^{n_2})} \leq C
|\Omega|^{-{1\over 2}}$.

The similar argument as above shows that for every  $0\leq k_1,k_2\leq M$,

\begin{eqnarray*}
&&\hspace{-1cm}\Big|\Big\langle  \Big( (\ell(I_R)^2L_1)^{k_1}\otimes (\ell(J_R)^2L_2)^{k_2}\Big)b_R,g\Big\rangle\Big|\\[2pt]
&=&\bigg|\int\limits_{\mathbb{R}^{n_1}\times\mathbb{R}^{n_2}}\int\limits_0^{3\ell(I_R)} \int\limits_0^{3\ell(J_R)}\\
&&\hskip.5cm\Big( (\ell(I_R)^2L_1)^{k_1}\otimes
(\ell(J_R)^2L_2)^{k_2}\Big)t_1^{2M}\Phi(t_1\sqrt{L_1})t_2^{2M}\Phi(t_2\sqrt{L_2})\big(A_R(\cdot,t)\big)(x)g(x){dxdt
\over
t_1  t_2}  \bigg|\\[2pt]
& \leq&
 \ell(I_R)^{2M}\ell(J_R)^{2M} \int\limits_{\mathbb{R}^{n_1}\times\mathbb{R}^{n_2}}
\int\limits_0^\infty\!\!\int\limits_0^\infty\big|A_R(x,t)\big|
\big|(t_1^2L_1)^{k_1}
\Phi(t_1\sqrt{L_1})(t_2^2L_2)^{k_2}\Phi(t_2\sqrt{L_2})\big(g\big)(x)\big|
{dxdt \over
t_1  t_2}   \\[2pt]
& \leq&  C \ell(I_R)^{2M}\ell(J_R)^{2M}
\|A_R\|_{L^2(dxdt/(t_1t_2))},
\end{eqnarray*}

\noindent which gives
\begin{eqnarray*}
&&\hspace{-1cm}\sum_{R\subset\Omega}\ell(I_R)^{-4M}\ell(J_R)^{-4M}
\|\big(\ell(I_R)^2L_1\big)^{k_1}\otimes\big(\ell(J_R)^2L_2\big)^{k_2}b_R\|^2_{L^2(\mathbb{R}^{n_1}\times\mathbb{R}^{n_2})} \\[2pt]
&&\ \ \leq  C
\sum_{R\subset\Omega}\|A_R\|^2_{L^2(dxdt/(t_1t_2))}
  \leq  C |\Omega|^{-1} .
\end{eqnarray*}

 Combining all the estimates above, we can see that $a$ is
a $(H^1_{L_1, L_2}, 2, M)$-atom as in Definition \ref{def of product atom}  up to some
constant depending only on $M,\psi$. This completes the proof of
Lemma \ref{lemma-of-product operator-pi L M}.
\end{proof}

We are now ready to establish the atomic decomposition of
$H_{L_1,L_2}^1(\mathbb{R}^{n_1}\times\mathbb{R}^{n_2})\cap
L^2(\mathbb{R}^{n_1}\times\mathbb{R}^{n_2})$.

\begin{prop}\label{prop-product H-SL subset H-at}
Suppose $M\geq 1$. If $f\in
H_{L_1,L_2}^1(\mathbb{R}^{n_1}\times\mathbb{R}^{n_2})\cap
L^2(\mathbb{R}^{n_1}\times\mathbb{R}^{n_2})$, then there exist a
family of $(H^1_{L_1, L_2}, 2, M)$-atoms $\{a_j\}_{j=0}^\infty$ and a sequence of
numbers $\{\lambda_j\}_{j=0}^\infty\in \ell^1$ such that $f$ can be
represented in the form $f=\sum\lambda_ja_j$, with the sum
converging in $L^2(\mathbb{R}^{n_1}\times\mathbb{R}^{n_2})$, and
$$ \|f\|_{\mathbb{H}^1_{L_1,L_2,at,M}(\mathbb{R}^{n_1}\times\mathbb{R}^{n_2})}\leq C\sum_{j=0}^\infty|\lambda_j|
\leq
C\|f\|_{H_{L_1,L_2}(\mathbb{R}^{n_1}\times\mathbb{R}^{n_2})},
$$

\noindent where $C$ is independent of $f$. In particular,
$$ H_{L_1,L_2}^1(\mathbb{R}^{n_1}\times\mathbb{R}^{n_2})\cap
L^2(\mathbb{R}^{n_1}\times\mathbb{R}^{n_2})\ \ \subseteq\
\mathbb{H}^1_{L_1,L_2,at,M}(\mathbb{R}^{n_1}\times\mathbb{R}^{n_2}).
$$

\end{prop}

\begin{proof}

Let $f\in
H_{L_1,L_2}^1(\mathbb{R}^{n_1}\times\mathbb{R}^{n_2})\cap
L^2(\mathbb{R}^{n_1}\times\mathbb{R}^{n_2})$, and set
$$ F(\cdot,t)=(t_1^2L_1e^{-t_1^2L_1}\otimes t_2^2L_2e^{-t_2^2L_2})f. $$

\noindent We note that $F\in
T^{1,2}(\mathbb{R}^{n_1}\times\mathbb{R}^{n_2})\cap
T^{2,2}(\mathbb{R}^{n_1}\times\mathbb{R}^{n_2})$ by the definition
of $H_{L_1,L_2}^1(\mathbb{R}^{n_1}\times\mathbb{R}^{n_2})$ and Lemma
\ref{lemma-of-product operator-pi L M}. Therefore, by Proposition
\ref{prop-of-product tent space},
$$ F=\sum_{j=0}^\infty\lambda_jA_j, $$

\noindent where each $A_j$ is a $T^{1,2}$-atom, the sum converges in
both $T^{1,2}(\mathbb{R}^{n_1}\times\mathbb{R}^{n_2})$ and
$T^{2,2}(\mathbb{R}^{n_1}\times\mathbb{R}^{n_2})$, and

\begin{eqnarray}\label{e1 in section 5.3.3}
\sum_{j=0}^\infty|\lambda_j|\leq
C\|F\|_{T^{1,2}(\mathbb{R}^{n_1}\times\mathbb{R}^{n_2})}=C\|f\|_{H_{L_1,L_2}^1(\mathbb{R}^{n_1}\times\mathbb{R}^{n_2})}.
\end{eqnarray}

\noindent From the spectral theory (\cite{Yo}), we have

\begin{eqnarray}\label{e2 in section 5.3.3}
f(x)&=& \int_0^\infty\int_0^\infty
\psi(t_1\sqrt{L_1})\psi(t_2\sqrt{L_2})(t_1^2e^{-t_1L_1}\otimes
t_2^2e^{-t_2L_2})f(x){dt\over t_1t_2}\\
&=&  \pi_{L_1,L_2}(F)(x)=  c_{\psi}
\sum_{j=0}^\infty\lambda_j\pi_{L_1,L_2}(A_j)(x),\nonumber
\end{eqnarray}

\noindent where   the last sum converges in
$L^2(\mathbb{R}^{n_1}\times\mathbb{R}^{n_2})$. Moreover, by Lemma
\ref{lemma-of-product operator-pi L M}, for every $M\geq 1$, we have
that up to multiplication by some harmless constant $C_M$, each
$a_j=c_{\psi}\pi_{L_1,L_2}(A_j)$ is a $(H^1_{L_1, L_2}, 2,
M)$-atom. Consequently, the last sum in (\ref{e2 in section 5.3.3})
is an atomic representation, so that $f\in
\mathbb{H}^1_{L_1,L_2,at,M}(\mathbb{R}^{n_1}\times\mathbb{R}^{n_2})$,
and by (\ref{e1 in section 5.3.3}) we have

$$\|f\|_{\mathbb{H}_{L_1,L_2,at,M}^1(\mathbb{R}^{n_1}\times\mathbb{R}^{n_2})}
\leq
C\|f\|_{H_{L_1,L_2}^1(\mathbb{R}^{n_1}\times\mathbb{R}^{n_2})}.
$$

\noindent
This concludes the proof of Theorem \ref{theorem of Hardy space atomic decom}
\end{proof}

\vskip 1cm

\section{Boundedness of singular integrals on  Hardy space  $H^1_{L_1,
L_2}(\mathbb{R}^{n_1}\times\mathbb{R}^{n_2})$ and Lebesgue spaces
$L^p(\mathbb{R}^{n_1}\times\mathbb{R}^{n_2})$}
\setcounter{equation}{0}

Let $T$ be a bounded linear operator
 on $L^2(\mathbb{R}^{n_1}\times\mathbb{R}^{n_2})$ with an associated kernel $K(x_1,y_1,x_2,y_2)$ in the sense that

\begin{eqnarray}\label{def of singular integral operators}
Tf(x_1,x_2)=\iint_{\mathbb{R}^{n_1}\times\mathbb{R}^{n_2}}
K(x_1,y_1,x_2,y_2)f(y_1,y_2)dy_1dy_2,
\end{eqnarray}

\noindent
and the above formula holds for each continuous function $f$ with compact support, and for almost all $(x_1, x_2)$
not in the support of $f$. We use the same definitions for ${\widetilde K}^{(1)} (x_1, y_1) (x_2, y_2)$,
${\widetilde K}^{(2)}(x_2, y_2) (x_1, y_1)$, $T\circ(e^{-t_1L_1}\otimes e^{-t_2L_2}),$
${\widetilde K}^{(1)}_{(t_1, t_2)} (x_1, y_1) (x_2, y_2)$,
${\widetilde K}^{(2)}_{(t_1, t_2)} (x_2, y_2) (x_1, y_1)$ and $\Delta K_{(t_1,\, t_2)}(x_1,y_1,x_2,y_2)$ as in the Introduction.
%
%
%
%
%
%


\medskip

\noindent

The aim of this section is to prove the following  theorem.

 \medskip

\begin{theorem}\label{theorem T from Hp to Lp}
Let $T$ be a bounded linear operator
 on $L^2(\mathbb{R}^{n_1}\times\mathbb{R}^{n_2})$ with an associated kernel $K(x_1,y_1,x_2,y_2)$.
Suppose that the composite operators   $T\circ(e^{-t_1L_1}\otimes
e^{-t_2L_2}), t_i\geq 0, i=1,2$ have associated kernels
  $K_{(t_1, t_2)}(x_1,y_1,x_2,y_2)$ in the sense of  \eqref{def of singular integral operator}  and there exist    constants $\delta>0$ and $C>0$ such
that for all $\gamma_1,\gamma_2\geq 2$,

\begin{eqnarray} \label{cond1-}
 \int_{|x_1-y_1|>\gamma_1t_1}
\|{\widetilde K}^{(1)}{(x_1,y_1)}-{{\widetilde K}_{(t_1^2, \, 0)}^{(1)}}{(x_1,y_1)}\|_{ (L^2({\mathbb R}^{n_2}),\, L^2({\mathbb R}^{n_2}) )}dx_1\leq
C\gamma_1^{-\delta},
\end{eqnarray}

\begin{eqnarray}\label{cond2-}
 \int_{|x_2-y_2|>\gamma_2t_2}
\|{\widetilde K}^{(2)}{(x_2,y_2)}-{{\widetilde K}_{(0, \, t_2^2)}^{(2)}}{(x_2,y_2)}\|_{ (L^2({\mathbb R}^{n_1}),\, L^2({\mathbb R}^{n_1}) )}dx_2\leq
C\gamma_2^{-\delta},
\end{eqnarray}
\begin{eqnarray} \label{cond3-}
 \int_{\substack{|x_1-y_1|>\gamma_1t_1\\ |x_2-y_2|>\gamma_2t_2}}\big|\Delta K_{(t_1^2,t_2^2)}(x_1,y_1,x_2,y_2)\big|dx_1dx_2\leq
C\gamma_1^{-\delta}\gamma_2^{-\delta}.
\end{eqnarray}

Then $T$ extends to a bounded operator from
$H_{L_1, L_2}^1(\mathbb{R}^{n_1}\times\mathbb{R}^{n_2})$ to
$L^1(\mathbb{R}^{n_1}\times\mathbb{R}^{n_2})$. Hence,  $T$ can be extended from $L^2(\mathbb{R}^{n_1}\times\mathbb{R}^{n_2})\cap
L^p(\mathbb{R}^{n_1}\times\mathbb{R}^{n_2})$ to a bounded operator on
$L^p(\mathbb{R}^{n_1}\times\mathbb{R}^{n_2})$   for all $1<p\leq 2$.
\end{theorem}

{\bf Remark}: Conditions \ref{cond1-}, \ref{cond2-} and \ref{cond3-} are the same as conditions
\ref{condn1-}, \ref{condn2-} and \ref{condn3-} in the Introduction, respectively.  In the above Theorem \ref{theorem T from Hp to Lp},
we repeat these conditions only for clarity.
\medskip

\begin{proof}
To prove Theorem \ref{theorem T from Hp to Lp}, it suffices to show
that   $T$  is uniformly bounded on each $(H^1_{L_1, L_2}, 2, M)$ atom $a$ with
$M>\max({n_1/4},{n_2/4})$, and  there exists a constant
$C>0$ independent of $a$ such that

\begin{eqnarray}\label{T alpha uniformly bd}
\|T(a)\|_{L^1(\mathbb{R}^{n_1}\times\mathbb{R}^{n_2})}\leq C.
\end{eqnarray}

From the definition of $(H^1_{L_1, L_2}, 2, M)$ atom, it follows
that $a$ is supported in some $\Omega\subset
\mathbb{R}^{n_1}\times\mathbb{R}^{n_2}$ and $a$ can be further
decomposed into $a=\sum_{R\in m(\Omega)}a_R$. For any $R=I \times
J\subset\Omega$, let $l$ be the biggest dyadic cube containing  $I$,
so that $l\times J\subset\widetilde{\Omega}$, where
$\widetilde{\Omega}=\{x\in\mathbb{R}^{n_1}\times\mathbb{R}^{n_2}:\
M_s(\chi_{\Omega})(x)>1/2\}$. Next, let $Q$ be the biggest dyadic
cube containing $J$, so that $l\times Q\subset
\widetilde{\widetilde{\Omega}}$, where
$\widetilde{\widetilde{\Omega}}=\{x\in
\mathbb{R}^{n_1}\times\mathbb{R}^{n_2}:\
M_s(\chi_{\widetilde{\Omega}})(x)>1/2\}$. Now let $\widetilde{R}$ be
the 100-fold dilate of $l\times Q$ concentric with $l\times Q$.
Clearly, an application of the strong maximal function theorem shows
that $\big|\bigcup\limits_{R\subset\Omega} \widetilde{R}\big|\leq
C|\widetilde{\widetilde{\Omega}}|\leq C|\widetilde{\Omega}|\leq
C|\Omega|$. From (iii) in the definition of $(H^1_{L_1, L_2}, 2, M)$
atom, we can obtain that

\begin{eqnarray}\label{T alpha uniformly bd on inside of Omega}
\int_{\cup \widetilde{R}}|T(a)(x)|dx &\leq& \big|\cup
\widetilde{R}\big|^{1/2}\|T(a)\|_{L^2}\leq
C|\Omega|^{1/2}
\|a\|_{L^2}
 \leq   C|\Omega|^{1/2}|\Omega|^{-1/2}
 \leq  C.\nonumber
\end{eqnarray}

\noindent
Therefore,  the proof of  (\ref{T alpha uniformly bd}) reduces to show
that
\begin{eqnarray}\label{T alpha uniformly bd on outside of Omega}
\int_{\big(\cup \widetilde{R}\big)^c}|T(a)(x)|dx \leq C.
\end{eqnarray}

\noindent Since  $a=\sum_{R\in m(\Omega)}a_R$, we have
\begin{eqnarray*}
\int_{\big(\cup \widetilde{R}\big)^c}|T(a)(x)|dx&\leq& \sum_{R\in
m(\Omega)} \int_{{\widetilde{R}}^c}|T(a_R)(x)|dx\\
&\leq& \sum_{R\in
m(\Omega)}\int_{(100l)^c\times\mathbb{R}^{n_2}}|T(a_R)(x)|dx+\sum_{R\in
m(\Omega)}\int_{\mathbb{R}^{n_1}\times(100S)^c}|T(a_R)(x)|dx \\
&=&D+E.
\end{eqnarray*}

\noindent
For term $D$, we observe that
\begin{eqnarray*}
\int_{(100l)^c\times\mathbb{R}^{n_2}}|T(a_R)(x)|dx&=&\Big(\int_{(100l)^c\times 100J} +\int_{(100l)^c\times (100J)^c}\Big) |T(a_R)(x)|dx
 =D_1+D_2.
\end{eqnarray*}

\noindent
Let us first estimate the term $D_1$. In what follows, we let $t_1=\ell(I)$, $t_2=\ell(J)$ and denote by
$
a_{R,1}=(L_1\otimes 1\!\!1_2 ) b_R $ and $ a_{R,2}=(1\!\!1_1 \otimes L_2) b_R.
$
It follows that

\begin{eqnarray}\label{ea}
 a_R=(L_1\otimes L_2 ) b_R=(1\!\!1_1\otimes L_2) a_{R,1}=(L_1\otimes 1\!\!1_2)a_{R,2}.
\end{eqnarray}

\noindent For every $x_1\in{\mathbb R}^{n_1}$, we have the identity:

\begin{eqnarray}\label{eID}
g(x_1)&=& \Big(  t_1^{-2}  \int_0^{\sqrt{2}t_1}sds \Big)\cdot g(x_1)\nonumber\\
&=& t_1^{-2}  \int_0^{\sqrt{2}t_1}s(1\!\!1_1-e^{-s^2L_1})g(x_1)ds +t_1^{-2} \int_0^{\sqrt{2}t_1}se^{-s^2L_1}g(x_1)ds\nonumber\\
&=& t_1^{-2}  \int_0^{\sqrt{2}t_1}s(1\!\!1_1-e^{-s^2L_1})g(x_1)ds +2t_1^{-2}(1\!\!1_1-e^{-2t_1^2L_1}) L_1^{-1}g(x_1).
\end{eqnarray}

\noindent   This, in combination with the fact that $  (L_1^{-1}\otimes 1\!\!1_2 ) a_R=a_{R,2}$,  gives the term $D_1$ in the following way

\begin{eqnarray*}
D_1&\leq &  t_1^{-2}\int_0^{\sqrt{2}t_1}
\iint_{(100l)^c\times 100J} s\big|
T\circ\big((1\!\!1_1-e^{-s^2L_1})\otimes 1\!\!1_2\big)(a_{R})(x_1,x_2) \big|dx_1dx_2  ds\\[2pt]
&&+ Ct_1^{-2}
\iint_{(100l)^c\times 100J}\big|
T\circ\big((1\!\!1_1-e^{-2t_1^2L_1}) \otimes 1\!\!1_2\big) (a_{R,2})(x_1,x_2) \big|dx_1dx_2 \\[2pt]
&=& D_{11} + D_{12}.
\end{eqnarray*}

\noindent

\noindent Fix $y_1$ so that
$a_{R}(y_1,\cdot)$ is supported on $10J$. Hence one may apply   H\"older's  inequality  to the term $D_{11}$ to obtain

\begin{eqnarray*}
D_{11}
 &\leq&  t_1^{-2}\int\limits_0^{\sqrt{2}t_1}\!\!
\iint\limits_{(100l)^c\times 10I}\!\!s \,
\big\|\big\{{\widetilde K}^{(1)}{(x_1,y_1)} -{{\widetilde K}_{(s^2,\, 0)}}{(x_1,y_1)}\big\}a_{R}(y_1,\cdot)\big\|_{L^1(|x_2|\sim 100|J|; dx_2)}
  dy_1dx_1  ds\\
  &\leq&  C  t_1^{-2}  |J|^{1/2} \int\limits_0^{\sqrt{2}t_1}\!\!
\iint\limits_{(100l)^c\times 10I}\!\!s \,
\|{\widetilde K}^{(1)}{(x_1,y_1)}-{{\widetilde K}_{(s^2, \, 0)}}{(x_1,y_1)}\|_{(L^2,\, L^2)}\\
&&\hskip.5cm\times\|a_{R}(y_1, \cdot)\|_{L^2(dy_2)}\ dy_1dx_1  ds.
\end{eqnarray*}

\noindent Noting that $|x_1-y_1|\geq 50\ell(l)\geq 50 {\ell(l)\over
\ell(I)}\cdot \ell(I)=50 {\ell(l)\over \ell(I)} \cdot t_1 \geq
{50\over \sqrt{2}} {\ell(l)\over \ell(I)} \cdot s$,
we have
\begin{eqnarray*}
D_{11} &\leq & C   \gamma_1(R)^{-\delta}  t_1^{-2}  |J|^{1/2}   \int_0^{\sqrt{2}t_1}sds\
\int_{10I}\|a_{R}(y_1, \cdot)\|_{L^2(dy_2)}\ dy_1\\
&\leq & C     \gamma_1(R)^{-\delta}   |J|^{1/2}  |I|^{1/2}
\Big(\int_{10I}\|a_{R}(y_1, \cdot)\|^2_{L^2(dy_2)}\ dy_1\Big)^{1/2}\\
 &\leq&   C\gamma_1(R)^{-\delta} |R|^{1/2}\|a_{R}\|_{L^2}.
\end{eqnarray*}

\noindent
\noindent The similar argument as above shows that

\begin{eqnarray*}
 D_{12}
 &\leq& C \ell(I)^{-2}|J|^{1/2} \int\limits_{(100l)^c}\int\limits_{10I}\
\|{\widetilde K}^{(1)}{(x_1,y_1)}-{{\widetilde K}^{(1)}_{(2t_1^2, \, 0)}}{(x_1,y_1)}\|_{(L^2\rightarrow L^2)} \\
&&\hskip.5cm\times\|(a_{R,2})(y_1, \cdot)\|_{L^2(dy_2)}\ dy_1dx_1\\[5pt]
&\leq& C \ell(I)^{-2}|J|^{1/2}\gamma_1(R)^{-\delta}
\int_{3I}\| a_{R,2}(y_1, \cdot)\|_{L^2(dy_2)}dy_1\\[5pt]
&\leq& C
\ell(I)^{-2}|R|^{1/2}\gamma_1(R)^{-\delta}\| a_{R,2}\|_{L^2}.
\end{eqnarray*}

\noindent From   estimates of $D_{11}$ and $D_{12}$, we use the
property of $(H^1_{L_1, L_2}, 2, M)$ atoms   and
Journ\'e's covering lemma to get

\begin{eqnarray*}\label{estimate of I1}
 \lefteqn{\sum_{R\in m(\Omega)}\int_{(100l)^c\times
100J}|T(a_R)(x)|dx}\\
&\leq& C \sum_{R\in m(\Omega)} \gamma_1(R)^{-\delta} |R|^{1/2}\Big(\|a_{R}\|_2
+\ell(I)^{-2}\| a_{R,2}\|_{L^2}\Big)\nonumber\\
&\leq& C\Big(\sum_{R\in
m(\Omega)}\gamma_1(R)^{-2\delta} |R| \Big)^{1/2}\bigg\{ \Big(\sum_{R\in
m(\Omega)}\|a_{R}\|_2^2\Big)^{1/2} +\Big(\sum_{R\in
m(\Omega)}\ell(I)^{-4}\| a_{R,2}\|_{L^2}^2\Big)^{1/2}\bigg\}\nonumber\\
&\leq & C |\Omega|^{-1/2} |\Omega|^{1/2}
 \leq   C.\nonumber
\end{eqnarray*}

Consider the term $D_{2} $. We use  an argument in (\ref{eID}) and the fact that $a_R=(L_1\otimes L_2 ) b_R=
(1\!\!1_1\otimes L_2) a_{R,1}=(L_1\otimes 1\!\!1_2)a_{R,2}$ to obtain

 \begin{eqnarray*}\label{eID12}
 \hspace{0.3cm} D_2\!&=&\!\!\int\limits_0^{\sqrt{2}t_1}\int\limits_0^{\sqrt{2}t_2}\!\!\!\iint\limits_{(100l)^c\times (100J)^c}\!\!\!
 s_1 s_2  \big|T\circ\big( (1\!\!1_1-e^{-s_1^2L})\otimes (1\!\!1_2-e^{-s_2^2L})\big)(a_{R})(x_1,x_2) \big|dx_1dx_2
{ds_1ds_2 \over  t_1^2 t_2^{2}} \nonumber\\
&+& {  2 t_2^{-2}} \int\limits_0^{\sqrt{2}t_1}\!\!\iint\limits_{(100l)^c\times (100J)^c}\!\! s_1
\big|T\circ\big( (1\!\!1_1-e^{-s_1^2L}) \otimes  (1\!\!1_2-e^{-2t_1^2L})  \big)(a_{R, 2})(x_1,x_2) \big|dx_1dx_2 {ds_1  \over  t_1^2  }\nonumber\\
&+& { 2t_1  ^{-2}} \int\limits_0^{\sqrt{2}t_2}\!\!\iint\limits_{(100l)^c\times (100J)^c} \!\!s_2\big|T\circ\big(  (1\!\!1_1-e^{-2t_1^2L})
  \otimes (1\!\!1_2-e^{-s_2^2L})\big)(a_{R, 1})(x_1,x_2) \big|dx_1dx_2{ ds_2 \over   t_2^{2}}\nonumber\\
&+& {  4(t_1 t_2)^{-2}} \iint\limits_{(100l)^c\times (100J)^c}\!\!
\big|T\circ\big( (1\!\!1_1-e^{-2t_1^2L})  \otimes (1\!\!1_2-e^{-2t_2^2L})  \big)(b_{R})(x_1,x_2) \big|dx_1dx_2 \nonumber\\
&=& D_{21}+D_{22}+D_{23}+D_{24}.
\end{eqnarray*}

\noindent To estimate the term $D_{21}$, we use the use condition (\ref{cond3-}) to obtain

\begin{eqnarray*}
 D_{21}
&=& \int\limits_0^{\sqrt{2}t_1} \int\limits_0^{\sqrt{2}t_2}\!\!\!\iint\limits_{(100l)^c\times (100J)^c}\!\!\!
  s_1s_2 \Big|\int_{3R}
 \Delta K_{(s_1^2, s_2^2)}(x_1,y_1,x_2,y_2)
a_{R}(y_1,y_2)dy_1dy_2 \Big|dx_1dx_2\  {ds_1ds_2 \over  t_1^2 t_2^{2}}\\[5pt]
&\leq & \int\limits_0^{\sqrt{2}t_1} \int\limits_0^{\sqrt{2}t_2}\!\!\!\int_{3R}
\Big(\int_{\substack{|x_1-y_1|>\gamma_1t_1\\ |x_2-y_2|>\gamma_2t_2}}\big|\Delta K_{(t_1^2,t_2^2)}(x_1,y_1,x_2,y_2)\big|dx_1dx_2\Big)
 \big|a_{R}(y_1,y_2)\big| dy_1dy_2   {ds_1ds_2 \over  t_1^2 t_2^{2}}\\[5pt]
&\leq & C\gamma_1(R)^{-\delta}\ \gamma_2(R)^{-\delta}   \big\|a_{R}\big\|_{L^1(3R)}   \int\limits_0^{\sqrt{2}t_1}\!\!\int\limits_0^{\sqrt{2}t_2} s_1s_2
{ds_1ds_2 \over  t_1^2 t_2^{2}} \\[5pt]
&\leq & C\gamma_1(R)^{-\delta} |R|^{1/2} \|a_{R}\|_{L^2} \ \ \ \ \ \ ({\rm since }\ \gamma_2(R)\geq 1).
\end{eqnarray*}

\noindent The similar argument as above shows that

\begin{eqnarray*}
D_{22}+D_{23}+D_{24}&\leq&  C\gamma_1(R)^{-\delta} |R|^{1/2} \\[3pt]
&& \times  \Big(\ell(J)^{-2} \|a_{R,1}\|_{L^2}
+  \ell(I)^{-2}
\|a_{R,2}\|_{L^2}
 +\ell(I)^{-2}\ell(J)^{-2}
\|b_R\|_{L^2}\Big).
\end{eqnarray*}

\noindent
From the estimate of terms $D_{21}$, $D_{22}$, $D_{23}$ and $D_{24}$,   the H\"older inequality,
Journ\'e's covering lemma and the properties of $(H^1_{L_1, L_2}, 2, M)$ atoms, we have

\begin{eqnarray}\label{estimate of I2}
\sum_{R\in m(\Omega)}\int_{(100l)^c\times (100J)^c}|T(a_R)(x)|dx
\leq  C |\Omega|^{-1/2} |\Omega|^{1/2} \leq  C.
\end{eqnarray}

\noindent
Now, think of replacing $\Omega$ by ${\widetilde\Omega}$. For each $R=I\times J
 \in m(\Omega)$ there corresponds a new rectangle $R'=l\times J\in
 m_1({\widetilde\Omega})$. Since $Q$ is the longest dyadic interval containing $J$
 so that $l\times Q\subseteq {\widetilde \Omega}$, then

 \begin{eqnarray}
\int_{\mathbb{R}^{n_1}\times(100Q)^c}|T(a_R)(x)|dx\leq
c|R|^{1/2}\gamma^{-\delta}_2(R', {\widetilde \Omega})
\|a_R\|_{L^2(dydt/(t_1t_2))}. \label{e3.20}
\end{eqnarray}

\noindent
The claim is that $\sum_{R\in m(\Omega)}|R|\gamma^{-2\delta}_2(R', {\widetilde \Omega})
\leq C|\Omega|.$ This is so because if $R_1$ and $R_2$ are in
$m(\Omega)$ with $R'_1=R'_2$ then $R_1\cap R_2=\emptyset$ or
$R_1=R_2$. It follows that

\begin{eqnarray*}
\sum_{R\in m(\Omega)}|R|\gamma_2^{-2\delta}(R', {\widetilde
\Omega})&\leq& \sum_{\overline{R}\in m_1({\widetilde
\Omega})}\Big(\sum_{\overline{R}=R'}|R|\Big)
\gamma_2^{-2\delta}(\overline{R}, {\widetilde \Omega})\\
&\leq& \sum_{\overline{R}\in m_1({\widetilde \Omega})}|\overline{R}|
\gamma_2^{-2\delta}(\overline{R}, {\widetilde \Omega})\\
&\leq& C|{\widetilde \Omega}|\leq C|\Omega|.
\end{eqnarray*}

 \noindent
  Estimate (\ref{T alpha uniformly bd on outside of Omega}) is obtained
 and  then we obtain the proof
of Theorem \ref{theorem T from
Hp to Lp} for   $f\in H^1_{L_1, L_2}({\mathbb R}^{n_1}\times {\mathbb R}^{n_2}) \cap L^2({\mathbb R}^{n_1}\times {\mathbb R}^{n_2})$.
It then follows from
   a standard
argument  that  for any $f\in H^1_{L_1,L_2}({\mathbb R}^{n_1}\times
{\mathbb R}^{n_2})$, $f$  has an atomic  decomposition
(\ref{e3.12}).  See, for example, Chapter III of  \cite{St}.
 This completes the proof Theorem \ref{theorem T from
Hp to Lp}.

\end{proof}

\bigskip

\section{Applications: boundedness of double Riesz transforms associated to Schr\"odinger operators
and multivariable spectral multipliers }
\setcounter{equation}{0}

In this section we shall   deduce  endpoint  estimates  of    a class  of examples of singular integrals with non-smooth kernels
including the double Riesz transforms associated to Schr\"odinger operators  and  the Marcinkiewicz-type  multipliers
for non-negative self-adjoint operators on product spaces.

\bigskip

\noindent
{\bf 5.1. The double Riesz transforms associated to Schr\"odinger operators }
For every $i=1,2$, we let $V_i\in L^1_{\rm loc}({\mathbb R}^{n_i})$ be a
nonnegative function on ${\mathbb R}^{n_i}$. The
Schr\"odinger operator with potential $V_i$ is defined  by
\begin{eqnarray}
L_i=-\triangle_{n_i}+V_i(x)\ \ \ \ {\rm on}\ {\mathbb R}^{n_i}, \ \ \ n_i\geq 1.
\label{e6.8}
\end{eqnarray}
The operator   $L_i$ is a self-adjoint  positive
definite operator on $L^2({\mathbb R}^{n_i})$.
From the Feynman--Kac formula, it is well-known that the    kernel $p_{t}(x_i,y_i)$
of the semigroup
$e^{-tL_i}$  satisfies the estimate
\begin{eqnarray}
0\leq p_{t}(x_i,y_i)\leq {1\over{ (4\pi t)^{n_i/2} }} e^{-{{|x_i-y_i|^2}\over 4t}}.
\label{e6.9}
\end{eqnarray}
 However,
  unless  $V_i$ satisfies additional conditions, the heat kernel can be
a discontinuous function of the
space variables and the H\"older continuity estimates may fail to
hold. See, for example, \cite{Da}.

 Consider the double Riesz transform $T:=\nabla L_1^{-1/2}\otimes \nabla L_2^{-1/2}$
associated to the operators $L_i$. An alternative definition is

\begin{eqnarray}\label{e8.16}
Tf(x_1, x_2)={1\over 4 {\pi}}\int_0^{\infty}\int_0^{\infty}\big(\nabla_{x_1} e^{-t_1L_1}\otimes \nabla_{x_2}
e^{-t_2L_2}\big)f(x_1, x_2){dt_1 dt_2\over \sqrt{t_1 t_2}}.
\end{eqnarray}

\noindent
It was proved in   \cite{Si2}  that for every $i=1,2,$ the Riesz transform
$\nabla L_i^{-1/2}$ is bounded on $L^p({\Bbb R}^{n_i} )$ for $1<p\leq 2$. See also  \cite{CD, DOY, Ou}.
Hence, by using iteration argument, the double Riesz transform $T$
 is bounded on $L^p({\Bbb R}^{n_1}\times {\Bbb R}^{n_2})$ for $1<p\leq 2$.

 \medskip

\begin{theorem}\label{th6.0}  Assume that
$L_i=- \triangle_{n_i} +V_i, i=1,2$, where $V_i\in L^1_{\rm loc}({\Bbb R}^{n_i})$ is a non-negative
function on ${\Bbb R}^{n_i}$. Then
the double Riesz transform $\nabla L_1^{-1/2}\otimes \nabla L_2^{-1/2}$
extends to a bounded operator from
$H_{L_1,L_2}^1(\mathbb{R}^{n_1}\times\mathbb{R}^{n_2})$ to
$L^1(\mathbb{R}^{n_1}\times\mathbb{R}^{n_2})$.
\end{theorem}

\begin{proof} This theorem is  a direct consequence of Theorem 4.1 
by verifying
conditions  (\ref{cond1-}),  (\ref{cond2-}) and  (\ref{cond3-}).

Let us verify condition   (\ref{cond1-}).  Following notation  in Theorem 4.1,
 we assume that   $K(x_1, y_1, x_2, y_2)$ is  an associated kernel of
the double Riesz transform  $T=\nabla L_1^{-1/2}\otimes \nabla L_2^{-1/2}$
in the sense of (\ref{def of singular integral operator}). From the definition, we
 can  write  in the form:

\begin{eqnarray}\label{e6.99}
K(x_1, y_1, x_2, y_2)=K_1(x_1, y_1)\cdot K_2(x_2, y_2),
\end{eqnarray}

\noindent
  where  $K_i(x_i, y_i)$ is  an associated kernel of the   Riesz transform  $\nabla L_i^{-1/2}, i=1,2.$
  Consider  the composite operators
$T\circ(e^{-t_1L_1}\otimes e^{-t_2L_2}), t_i\geq 0,$ which  have associated kernels
  $K_{(t_1, t_2)}(x_1,y_1,x_2,y_2)$ in the sense of (\ref{def of singular integral operator}).
Let  $k_{ (t_1^2, 0)}(x_1, y_1)  $ be  an associated kernel of the composite operator  $\nabla L_1^{-1/2}\big(I-e^{-t_1^2 L_1}\big).$
Then,

\begin{eqnarray}\label{e6.88}\hspace{1.2cm}
 K(x_1, y_1, x_2, y_2)- K_{(t_1^2, 0)}(x_1,y_1,x_2,y_2)
=  k_{ (t_1^2, 0)}(x_1, y_1)\cdot K_2(x_2, y_2).
\end{eqnarray}

 \noindent
The proof is done if we show  that
%
%

\begin{eqnarray}\label{e6.999}
\int_{|x_1-y_1|>\gamma_1 t_1} \big|k_{ (t_1^2, 0)}(x_1, y_1)\big| d
x_1 \leq C\gamma_1^{-\delta},\ \ \ \forall \gamma_1\geq 2
\end{eqnarray}

\noindent
for some   constants $\delta>0$ and $C>0$,  and then using  the fact that the  Riesz transform $\nabla L_2^{-1/2}$
is  bounded on $L^2(\mathbb{R}^{n_2}),$  we have

\begin{eqnarray*}
 &&\hspace{-1.5cm}\int_{|x_1-y_1|>\gamma_1t_1}
\|{\widetilde K}^{(1)}{(x_1,y_1)}-{{\widetilde K}_{(t_1^2, \, 0)}^{(1)}}{(x_1,y_1)}\|_{ (L^2({\mathbb R}^{n_2}),\, L^2({\mathbb R}^{n_2}) )}dx_1\\
&\leq&   \|\nabla L_2^{-1/2}\|_{L^2(\mathbb{R}^{n_2})\rightarrow L^2(\mathbb{R}^{n_2})}
 \int_{|x_1-y_1|>\gamma_1 t_1} \big|k_{ (t_1^2, 0)}(x_1, y_1)\big| d x_1
\\
&\leq&
C\gamma_1^{-\delta},
\end{eqnarray*}

\noindent which  proves  condition   (\ref{cond1-}).

 Let us prove estimate  (\ref{e6.999}). Let $\widetilde{p}_t(x_1, y_1)$ denote  the
kernels of the semigroup   $t{d\over dt}e^{-tL_1}$. The Riesz transform associated to $L_1$ is given by

\begin{eqnarray*}
\nabla L_1^{-1/2}={1\over
 \sqrt{\pi}}\int_0^\infty \nabla e^{-sL_1} {ds\over \sqrt{s}}.
\end{eqnarray*}

\noindent
Therefore,

\begin{eqnarray*}
\nabla L_1^{-1/2}\big(I-e^{-t_1^2 L_1}\big)&=&{1\over
 \sqrt{\pi}}\int_0^\infty \nabla e^{-sL_1}\big(I-e^{-t_1^2
L_1}\big){ds\over \sqrt{s}}\\
&=& -{1\over  \sqrt{\pi}}\int_0^\infty \int_s^{s+t_1^2} \nabla
\big(u{d\over du}e^{-uL_1}\big){du\over u} {ds\over \sqrt{s}},
\end{eqnarray*}

\noindent
and then the kernel $k_{ (t_1^2, 0)}(x_1, y_1)  $   of the composite operator  $\nabla L_1^{-1/2}\big(I-e^{-t_1^2 L_1}\big) $ can be written
in the form:

\begin{eqnarray*}
k_{(t_1^2,0)}(x_1,y_1) &=& -{1\over
 \sqrt{\pi}}\int_0^\infty\int_s^{s+t_1^2} \nabla
\widetilde{p}_{u}(x_1,y_1){du\over u} {ds\over \sqrt{s}}.
\end{eqnarray*}

\noindent
Note that by Proposition 3.1 of \cite{DOY},

\begin{eqnarray}\label{e6.9999}
\int_{\mathbb{R}^n}|\nabla
\widetilde{p}_{u}(x,y)|^2e^{\beta{|x-y|^2\over u}}dx\leq
Cu^{-{n\over 2}-1}
\end{eqnarray}

\noindent
for some constants
$\beta>0$ and $C>0$. Let $M> {n\over 2}$. There exists a positive constant  $C$   depending only on $n,\beta$ such that
\begin{eqnarray}\label{e6.99991}
\int_{|x_1-y_1|>\gamma_1 t_1} e^{-\beta{|x_1-y_1|^2\over u}}dx_1\leq
C (\gamma_1 t_1)^{n-2M} u^M.
\end{eqnarray}

\noindent
Using estimates (\ref{e6.9999}), (\ref{e6.99991}) and  the  H\"older  inequality,   we have

\begin{eqnarray*}
&&\hspace{-1.5cm}\lefteqn{\int_{|x_1-y_1|>\gamma_1 t_1} \big|k_{ (t_1^2,\, 0)}(x_1,
y_1)\big| d x_1}\\
& \leq&  {1\over
2\sqrt{\pi}}\int_0^\infty\int_s^{s+t_1^2}\Big(\int_{|x_1-y_1|>\gamma_1
t_1} \big|\nabla \widetilde{p}_{t_1}(x_1,y_1)\big|^2
e^{\beta{|x_1-y_1|^2\over u}}dx_1\Big)^{1\over2}\\ [2pt]
&&\hspace{3cm}\times \Big(\int_{|x_1-y_1|>\gamma_1 t_1}
e^{-\beta{|x_1-y_1|^2\over u}}dx_1\Big)^{1\over2}{du\over u}
{ds\over \sqrt{s}} \\
& \leq&  C(\gamma_1 t_1)^{n-2M\over2}  \int_0^\infty\int_s^{s+t_1^2}
u^{-{n\over4}-{1\over2}+{M\over 2}} {du\over u} {ds\over \sqrt{s}}.
\end{eqnarray*}

\noindent
Now we divide the last term into

\begin{eqnarray*}
C(\gamma_1 t_1)^{n-2M\over2} \bigg(
\int_0^{t_1^2}+\int_{t_1^2}^\infty\bigg)\int_s^{s+t_1^2}
u^{-{n\over4}-{1\over2}+{M\over 2}} {du\over u} {ds\over \sqrt{s}}.
\end{eqnarray*}

\noindent
Let ${n/2}+1<M<{n/2}+2$. By using  elementary integration, we show that both of the above  terms
 are bounded by $C \gamma_1^{(n-2M)/2}$. This proves estimate
(\ref{e6.999}), and then the proof of  condition   (\ref{cond1-}) is complete.

The similar argument as above shows   condition (\ref{cond2-}).
For
condition  (\ref{cond3-}), it is very easy to
obtain by  an iteration argument, and we skip it here. Hence, the
proof of Theorem~\ref{th6.0} is finished.

 \end{proof}

 \bigskip

\noindent
{\bf 5.2.  General  multivariable spectral multipliers   }
Suppose that  $L_1$ and $L_2$ are  non-negative self-adjoint operators  such that the corresponding
heat kernels satisfy Gaussian bounds ${\rm (GE)}$.
Let us  explain  the definition of multivariable spectral multipliers. We consider two self-adjoint
operators  $L_i, i=1,2,$ acting on spaces $L^2(\mathbb{R}^{n_i})$. The tensor product operators  $L_1\otimes1\!\!1_2$
and $1\!\!1_1\otimes L_2$ act on $L^2(\mathbb{R}^{n_1}\times \mathbb{R}^{n_2})$, where $\mathbb{R}^{n_1}\times \mathbb{R}^{n_2}$
is the Cartesian product of $\mathbb{R}^{n_1}$ and $ \mathbb{R}^{n_2}$ with the product measure. To simplify
notation we will write $L_1$ and $L_2$ instead of   $L_1\otimes1\!\!1_2$ and $1\!\!1_1\otimes L_2$. Note that there is a
unique spectral decomposition $E$ such that for all Borel subsets $A\subset {\mathbb R}^2$, $E(A)$ is a projection
on $L^2(\mathbb{R}^{n_1}\times \mathbb{R}^{n_2})$ and such that for any Borel subsets $A_j\subset {\mathbb R}, j=1,2,$
one has

$$
E(A_1\otimes A_2)= E_{L_1}(A_1)\otimes E_{L_2}(A_2).
$$

\noindent
Hence for any function $F: {\mathbb R}^2\rightarrow {\mathbb C}$ one can define the operators $F(L_1, L_2)$
acting as operator on space $L^2(\mathbb{R}^{n_1}\times \mathbb{R}^{n_2})$ by the formula

\begin{eqnarray}\label{e7.1}
F(L_1, L_2)=\int_{{\mathbb R}^2} F(\lambda_1, \lambda_2)dE(\lambda_1, \lambda_2).
\end{eqnarray}

\noindent
A straightforward  variation of classical spectral theory argument shows that for any bounded Borel function
 $F: {\mathbb R}^2\rightarrow {\mathbb C}$   the operator $F(L_1, L_2)$ is continuous on
 $L^2(\mathbb{R}^{n_1}\times \mathbb{R}^{n_2})$ and its norm is bounded by $\|F\|_{L^{\infty}}$.
Assume that the operator $F(L_1, L_2)$  has  an associated kernel $K(x_1,y_1,x_2,y_2)$ in the sense that

\begin{eqnarray}\label{eff}
F(L_1, L_2)f(x_1,x_2)=\iint_{\mathbb{R}^{n_1}\times\mathbb{R}^{n_2}}
K(x_1,y_1,x_2,y_2)f(y_1,y_2)dy_1dy_2,
\end{eqnarray}

\noindent
and the above formula holds for each continuous function $f$ with compact support, and for almost all $(x_1, x_2)$
not in the support of $f$.

 In this section we are looking for necessary smoothness conditions on function $F$ so that the operators $F(L_1, L_2)$
 is bounded on Hardy space $H^1_{L_1,
L_2}(\mathbb{R}^{n_1}\times\mathbb{R}^{n_2})$   and   on $L {\rm
log}^+ L$. The condition on function $F$ which we use is a variant
of the differentiability condition in   Fourier multiplier result,
see \cite{H}.
 We shall be working with an auxiliary nontrivial function $\omega$ of
compact support. The choice of $\omega$ that will be in the statements
is not unique. Let $\omega$ be a $C_0^\infty(\mathbb{R})$ function
such that
\begin{eqnarray}\label{function phi}
{\rm supp}\, \omega \subseteq ({1\over 4},1)\ \ \ {\rm and}\ \ \
\sum_{\ell\in\mathbb{Z}}\omega(2^{-\ell}\lambda)=1 \ \ \ {\rm for\
all}\ \ \lambda>0.
\end{eqnarray}


\noindent
Set

$$
\eta_{1}(\lambda_1)=\omega(\lambda_1), \ \ \ \eta_{2}( \lambda_2)=\omega(\lambda_2), \ \ \ \
\eta_{1,2}(\lambda_1, \lambda_2)=\omega(\lambda_1)\omega(\lambda_2).
$$

\noindent
We define a family of dilations $\{\delta_{t_1, \, t_2}\}$ acting on functions  $F: {\mathbb R}^2\rightarrow {\mathbb C}$ by the formula
\begin{eqnarray}\label{eeta}
\delta_{(t_1, \, t_2)}F(\lambda_1, \lambda_2)=F(t_1\lambda_1, t_2\lambda_2).
\end{eqnarray}
Consider the following multiparameter Sobolev norm on functions $F$
defined on $\mathbb{R}^{n_1}\times\mathbb{R}^{n_2}=\{(x,y)\}$

$$
\|F\|_{W_{s_1,s_2}^2(\mathbb{R}^{n_1}\times\mathbb{R}^{n_2})}= \|(
I+\Delta_x)^{s_1/2}( I+\Delta_y)^{s_2/2}F\|_{ L^2(\mathbb{R}^{n_1}\times \mathbb{R}^{n_2})},
$$

\noindent where $\Delta_x$ and $\Delta_y$ are the standard Laplace
operators on $\mathbb{R}^{n_1}$ and $\mathbb{R}^{n_2}$,
respectively.  Now we formula our main spectral multiplier result.

\begin{theorem}\label{th6.1} Suppose that  $L_1$ and $L_2$ are  non-negative self-adjoint operators  such that the corresponding
heat kernels satisfy Gaussian bounds ${\rm (GE)}$.   Let $s_1> {n_1+1\over
2}$, $s_2>{n_2+1\over 2}$. Then for any
 bounded  Borel function $F$ such that

\begin{eqnarray}\label{ec}
C_{F,\phi,s}&=&\sup_{t_1>0}\|\eta_{1}\delta_{(t_1, 1)}F
\|_{W^{2}_{s_1,s_2}}
+\sup_{t_2>0}\|\eta_{2}\delta_{(1,t_2)}F \|_{W^{2}_{s_1,s_2}}\nonumber\\
&&\hspace{1.8cm} +\sup_{t_1, t_2>0}\|\eta_{1,2}\delta_{(t_1, t_2)}F
\|_{W^{2}_{s_1,s_2}} <\infty,
\end{eqnarray}

\noindent
the operator $F(L_1,L_2)$ extends to a bounded operator from
$H_{L_1,L_2}^1(\mathbb{R}^{n_1}\times\mathbb{R}^{n_2})$ to
$L^1(\mathbb{R}^{n_1}\times\mathbb{R}^{n_2})$. 

\end{theorem}

\medskip

\noindent
{\bf Remarks.}

\medskip

1) We assume that $L_1$ and $L_2$ are positive so $F(L_1, L_2)$ depends only on the restriction of $F$ to
$[0, \infty)^2$. However, it is easier to state Theorem~\ref{th6.1} if one considers functions $F: {\mathbb R}^2\rightarrow {\mathbb C}$.

\medskip

2)  Let $k= \big[{n_1 \over 2}\big]+1, \ell=\big[{n_2 \over 2}\big]+1$, where $[a]$ is the integer part of $[a]$.
It is not difficult to check that if
$F\in C^{k}({\mathbb R})\times C^{\ell}({\mathbb R})$ and

$$
\big|\partial_{\xi}^{\alpha} \partial_{\eta}^{\beta} F(\xi, \eta)\big|\leq C |\xi|^{-\alpha} |\eta|^{-\beta}
 $$

\noindent
where $0\leq\alpha\leq k$ and $0\leq\beta\leq \ell$, then $F$ satisfies condition (\ref{ec}), see \cite{H}, \cite{DOS} and \cite{Ou}.

\bigskip

The proof of Theorem~\ref{th6.1} will be achieved in several steps. First,
we note that for every $t_1>0,$   the composite operator $ F({L_1},{L_2})(I-e^{-t^2_1L_1})$ is  bounded
on  $L^2(\mathbb{R}^{n_1}\times \mathbb{R}^{n_2})$, and assume that it has an  associated kernel
 $K_{F({L_1},{L_2})(I-e^{-t^2_1L_1})}(x_1, y_1, x_2, y_2)$ in the sense of (\ref{eff}). Similarly, assume that  in the sense of (\ref{eff}),
  the composite operator  $F({L_1},{L_2})(I-e^{-t^2_2L_2})$  has
 an associated kernel     $K_{F({L_1},{L_2})(I-e^{-t^2_2L_2})}(x_1, y_1, x_2, y_2)$,    and the composite operator  $F({L_1},{L_2})(I-e^{-t^2_1L_1})(I-e^{-t^2_2L_2})  $
has an associated kernel  $K_{F({L_1},{L_2})(I-e^{-t^2_1L_1})(I-e^{-t^2_2L_2})}(x_1, y_1, x_2, y_2)$.
   Following (4.2), we set

 $$  {\widetilde {K}}^{(1)}_{F({L_1},{L_2})(I-e^{-t^2_1L_1})}
 (x_1,y_1)(x_2, y_2)=K_{F({L_1},{L_2})(I-e^{-t^2_1L_1})}(x_1, y_1, x_2, y_2)
 $$
 $$  {\widetilde {K}}^{(2)}_{F({L_1},{L_2})(I-e^{-t^2_2L_2})}
 (x_2,y_2)(x_1, y_1)=K_{F({L_1},{L_2})(I-e^{-t^2_2L_2})}(x_1, y_1, x_2, y_2).
 $$
    We now state the following proposition.

\medskip

\begin{prop}\label{prop6.2}  Assume that the assumptions of  Theorem~\ref{th6.1}
are satisfied. Then we have

\medskip

(i)\,
   Let $R_1>0, s_1>0.$
 Then for any $s_2>1/2$, there exist  constants $C=C(s_1, s_2, \epsilon)$ and $\eta>0$ such that
 for
all $\gamma_1 \geq 2$ and $t_1>0,$

\begin{eqnarray}\label{e6.5}\hspace{0.5cm}
\hspace{3cm}&&\hspace{-3.8cm}\int_{|x_1-y_1|>\gamma_1 t_1}\big\|{\widetilde {K}}^{(1)}_{F({L_1},{L_2})(I-e^{-t^2_1L_1})} (x_1,y_1)\big\|_{L^2({\mathbb R}^{n_2})\rightarrow
L^2({\mathbb R}^{n_2})}^2 \big(1+ R_1|x_1-y_1|\big)^{s_1}dx_1\nonumber\\[3pt]
&\leq& C\gamma_1^{-\eta}
R_1^{n_1}\min\big\{\big(t_1R_1\big)^{-\eta}, (t_1R_1)^2  \big\} \,
\|\delta_{(R_1, 1)} F\|^2_{W^{\infty}_{{s_1+1+\epsilon\over
2},{s_2\over 2}}}
\end{eqnarray}

\noindent
for   all Borel functions $F$
such that ${\rm supp}\ F\subseteq[0,R_1^2]\times[0,\infty)$.

\medskip

(ii)\,
 Let $R_2>0, s_2>0.$ Then for any $s_1>1/2$, there exist  constants $C=C(s_1, s_2, \epsilon)$ and $\eta>0$ such that   for
all $\gamma_2 \geq 2$ and $t_2>0,$

\begin{eqnarray}\label{e6.6}\hspace{0.5cm}
\hspace{3cm}&&\hspace{-3.8cm}\int_{|x_2-y_2|>\gamma_2 t_2}\big\|{\widetilde {K}}^{(2)}_{F({L_1},{L_2})(I-e^{-t^2_2L_2})} (x_2,y_2)\big\|_{L^2({\mathbb R}^{n_1})\rightarrow
L^2({\mathbb R}^{n_1})}^2 \big(1+ R_2|x_2-y_2|\big)^{s_2}dx_2\nonumber\\[3pt]
&\leq& C\gamma_2^{-\eta}
R_2^{n_2}\min\big\{\big(t_2R_2\big)^{-\eta}, (t_2R_2)^2  \big\} \,
\|\delta_{(1, R_2)} F\|^2_{W^{\infty}_{ {s_1\over
2},{s_2+1+\epsilon\over 2}} }
\end{eqnarray}

\noindent
for   all Borel functions $F$
such that ${\rm supp}\ F\subseteq [0, \infty)\times [0,R_2^2]$.

\medskip

(iii)\,
 Let $R_i>0, s_i>0, i=1,2.$ Then   there exist    constants $C=C(s_1, s_2, \epsilon)$ and $\eta>0$ such that
 for
all $\gamma_1, \gamma_2 \geq 2$ and $t_1, t_2>0,$

\begin{eqnarray}\label{e6.7}
&&\int_{\substack{|x_1-y_1|>\gamma_1t_1\\ |x_2-y_2|>\gamma_2t_2}}
\big|
K_{F({L_1},{L_2})(I-e^{-t^2_1L_1})(I-e^{-t^2_2L_2})} (x_1,
y_1, x_2,y_2) \big|^2
  \prod_{i=1}^2\big(1+ R_i|x_i-y_i|\big)^{s_i}dx_1dx_2\nonumber\\[3pt]
&&\leq C\gamma_1^{-\eta}\gamma_2^{-\eta}R_1^{n_1}R_2^{n_2}
\min\big\{\big(t_1R_1\big)^{-\eta},   (t_1R_1)^2  \big\}
\min\big\{\big(t_2R_2\big)^{-\eta},   (t_2R_2)^2
\big\}\nonumber\\
&& \hskip1.5cm \times\|\delta_{(R_1, R_2)} F\|^2_{W^{\infty}_{
{s_1+1+\epsilon\over 2},{s_2+1+\epsilon\over 2} } }
\end{eqnarray}

\noindent
for   all Borel functions $F$
such that ${\rm supp}\ F\subseteq [0, R_1^2]\times [0,R_2^2]$.
\end{prop}

\medskip

\begin{proof} Let us first prove (\ref{e6.5}).  Set $\Phi_{(t_1^2, 0)}(\lambda_1)=(1-e^{-t_1^2\lambda_1})$ and

$$G(\lambda_1, \lambda_2)=e^{\lambda_1} \delta_{(R_1^2, 1)}\big[F({\lambda_1}, {\lambda_2})\Phi_{(t^2_1, 0)}(\lambda_1)\big].
$$

\noindent
If ${\widehat  G^{(1)}}$ denotes the Fourier transform of $G(\lambda_1, \lambda_2)$ in
the variable $\lambda_1$, then by the Fourier inversion formula,

\begin{eqnarray}\label{e6.8}
F({L_1}, {L_2})(I-e^{-t_1^2L_1})&=&G( {L_1}/R_1^{2}, {L_2})e^{-L_1/R_1^{2}}\nonumber\\[3pt]
&=&{1\over 2\pi}\int_{{\mathbb R}} e^{(i\xi_1 -1)L_1/R_1^{2}} {\widehat  G^{(1)}}(\xi_1,  {L_2}) d\xi_1.
\end{eqnarray}

\noindent
Then the kernel  $K_{F({L_1},{L_2})(I-e^{-t_1^2L_1})} (x_1,y_1, x_2, y_2)$ satisfies

\begin{eqnarray}\label{e6.88}\ \  \hspace{1cm}
K_{F({L_1},{L_2})(I-e^{-t_1^2L_1})} (x_1,y_1, x_2, y_2)=
{1\over 2\pi}\int_{{\mathbb R}} p_{(1-i\xi_1)R_1^{-2}}(x_1, y_1) K_{{\widehat  G^{(1)}}(\xi_1,  {L_2})}(x_2, y_2) d\xi_1.
\end{eqnarray}

\noindent Notice that by Lemma 4.2 in \cite{DOS}, we have

\begin{eqnarray*}
\int_{|x_1-y_1|>r}   \big| p_{(1-i\xi_1)R_1^{-2}}(x_1, y_1)\big|^2    dx_1 \leq R_1^{n_1}
\exp\Big({-c\Big[{r R_1\over (1+|\xi_1|)}\Big]^2}\Big),
\end{eqnarray*}

\noindent
which gives

\begin{eqnarray*}
&&\hspace{-1cm}\int_{|x_1-y_1|>\gamma_1 t_1}   \big| p_{(1-i\xi_1)R_1^{-2}}(x_1, y_1)\big|^2  |x_1-y_1|^{s_1}   dx_1 \\
&\leq&\sum_{k\geq 0:  \,  k\geq  {\gamma_1t_1R_1\over (1+|\xi_1|)}-1 } \int_{k{ (1+|\xi_1|)\over R_1}  \leq  |x_1-y_1|\leq (k+1){ (1+|\xi_1|)\over R_1} }
  \big| p_{(1-i\xi_1)R_1^{-2}}(x_1, y_1)\big|^2  |x_1-y_1|^{s_1}   dx_1\\
&\leq& (1+|\xi_1|)^{s_2} R_1^{-s_1}\sum_{k\geq 0:  \,  k\geq
{\gamma_1t_1R_1\over (1+|\xi_1|)}-1 } (k+1)^{s_1} \int_{
|x_1-y_1|\geq k{ (1+|\xi_1|)\over R_1} }
  \big| p_{(1-i\xi_1)R_1^{-2}}(x_1, y_1)\big|^2      dx_1\\
  &\leq& (1+|\xi_1|)^{s_1} R_1^{-s_1}\sum_{k\geq 0:  \,  k\geq  {\gamma_1t_1R_1\over (1+|\xi_1|)}-1 }
(k+1)^{s_1} R_1^{n_1} e^{-ck^2}\\
&\leq& C\Big({\gamma_1t_1R_1\over  1+|\xi_1|
}\Big)^{-\eta}(1+|\xi_1|)^{s_1} R_1^{n_1} R_1^{-s_1}
\end{eqnarray*}
for some $\eta>0$. This, in combination with (\ref{e6.8}) and
(\ref{e6.88}),  gives

\begin{eqnarray}\label{e6.9}
A:&=& \bigg[\int_{|x_1-y_1|>\gamma_1 t_1}\big\|{\widetilde {K}}^{(1)}_{F({L_1},{L_2})(I-e^{-t_1^2L_1})} (x_1,y_1)\big\|_{L^2({\mathbb R}^{n_2})\rightarrow
L^2({\mathbb R}^{n_2})}^2 (1+R_1|x_1-y_1|)^{s_1}dx_1\bigg]^{1/2}\nonumber\\
&\leq &C\int_{{\mathbb R}}
\big\|{\widehat  G^{(1)}}(\xi_1,  {L_2})\big\|_{L^2({\mathbb R}^{n_2})\rightarrow
L^2({\mathbb R}^{n_2})}\times \nonumber\\[3pt]
&&\hspace{2cm} \times  \Big( \int_{|x_1-y_1|>\gamma_1 t_1}  \big| p_{(1-i\xi)R_1^{-2}}(x_1, y_1)\big|^2   (1+R_1|x_1-y_1|)^{s_1}dx_1\Big)^{1/2}d\xi_1\\
&\leq& CR_1^{n_1/2}\Big({\gamma_1t_1R_1 }\Big)^{-\eta/2}
\bigg[\int_{{\mathbb R}} \big\|{\widehat  G^{(1)}}(\xi_1,
{L_2})\big\|^2_{L^2({\mathbb R}^{n_2})\rightarrow L^2({\mathbb
R}^{n_2})}  \big(1+|\xi_1|^2\big)^{s_1+1+\epsilon\over 2}
d\xi_1\bigg]^{1/2}.\nonumber
\end{eqnarray}

\noindent
Note that $W^2_s({\mathbb R})\subseteq L^{\infty}({\mathbb R})$ whenever  $s>1/2.$ By the embedding theorem of Sobolev spaces, we have

\begin{eqnarray}\label{e6.10}
 \big\|{\widehat  G^{(1)}}(\xi_1,  {L_2})\big\|_{L^2({\mathbb R}^{n_2})\rightarrow
L^2({\mathbb R}^{n_2})} &\leq&  C\big\|{\widehat  G^{(1)}}(\xi_1, \xi_2)\big\|_{L^{\infty}(\xi_2)} \nonumber\\
&\leq& C\Big[ \int_{{\mathbb R}}
\big|{\widehat  G } (\xi_1, \xi_2)\big|^2    (1+|\xi_2|^2)^{s_2/2}  d\xi_2\Big]^{1/2}.
\end{eqnarray}

\noindent
Putting (\ref{e6.10}) into (\ref{e6.9}), we obtain

\begin{eqnarray}\label{e6.11}
A &\leq& CR_1^{n_1/2}\Big({\gamma_1t_1R_1 }\Big)^{-\eta/2}
\bigg[\iint_{{\mathbb R}^2} \big|{\widehat  G } (\xi_1,
\xi_2)\big|^2   \big(1+|\xi_1|^2\big)^{s_1+1+\epsilon\over 2} \big(1
+|\xi_2|^2\big)^{s_2\over 2}
   d\xi_1d\xi_2\bigg]^{1/2}
\nonumber\\[3pt]
&\leq&CR_1^{n_1/2}\Big({\gamma_1t_1R_1 }\Big)^{-\eta/2}
\|G\|_{W^2_{{s_1+1+\epsilon\over 2},{s_2\over2}}}.
\end{eqnarray}

 \noindent
 However, supp $F\subseteq [0, R_1^2]\times [0, \infty)$ and supp $\delta_{(R_1^2, 1)} F\subseteq [0, 1]\times [0, \infty)$ so
 if $k$ is an integer greater than ${s_1+1+\epsilon\over 2},$

 \begin{eqnarray*}
 \|G\|_{W^2_{{s_1+1+\epsilon\over 2},{s_2\over2}}}&\leq&
 \|\delta_{(R_1^2, 1)}\big[F({\lambda_1}, {\lambda_2})\Phi_{(t_1^2, 0)}(\lambda_1)\big]
 \|_{W^2_{{s_1+1+\epsilon\over 2},{s_2\over2}}}\\
 &\leq&
 C\|\delta_{(R_1^2, 1)} F\|_{W^2_{{s_1+1+\epsilon\over 2},{s_2\over2}}}\times \|\delta_{(R_1^2, 1)} (1-e^{-t_1^2\cdot})\|_{C^k([0,1])}\\
 &\leq& C{(R_1 t_1)^2\over 1+(R_1 t_1)^2}
  \|\delta_{(R_1^2, 1)} F\|_{W^{2}_{{s_1+1+\epsilon\over 2},{s_2\over2}}}.
\end{eqnarray*}

 \noindent
 Therefore,
 \begin{eqnarray*}
A &\leq&CR_1^{n_1/2}\Big({\gamma_1t_1R_1 }\Big)^{-\eta/2} {(R_1
t_1)^2\over 1+(R_1 t_1)^2}
  \|\delta_{(R_1^2, 1)} F\|_{W^{2}_{{s_1+1+\epsilon\over 2},{s_2\over2}}}\\
  &\leq&CR_1^{n_1/2} \gamma_1^{-\eta/2} \min\big\{\big(t_1R_1\big)^{-\eta/2},   t_1R_1  \big\}  \,
  \|\delta_{(R_1^2, 1)} F\|_{W^{2}_{{s_1+1+\epsilon\over 2},{s_2\over2}}}.
\end{eqnarray*}

 \noindent
 This proves (\ref{e6.5}).

 The proof of (\ref{e6.6}) is similar to that of (\ref{e6.5}). For the proof of (\ref{e6.7}), we can obtain it by making
 minor  modifications with Lemma 3.5 of \cite{DOS}, and so we skip it.
\end{proof}

\bigskip

\noindent
\begin{proof}[Proof of Theorem~\ref{th6.1}.]  Let $F: {\mathbb R}^n\rightarrow {\mathbb C}$ be a bounded Borel function such that
 condition (\ref{ec}) holds.   To prove Theorem 6.1, it suffices to verify the assumptions of
 Theorems 4.1 and 5.1 for   $T=F( {L_1}, {L_2})$, i.e.,
there exists some constants $\delta>0$ and $C>0$ such
that for all $\gamma_1,\gamma_2\geq 2$,

\begin{eqnarray}  \label{e6.14}
 \int_{|x_1-y_1|>\gamma_1 t_1}
\|{\widetilde K}^{(1)}_{F({L_1}, {L_2})(I-e^{-t_1^2
{L_1}})}{(x_1,y_1)}\|_{ L^2({\mathbb R}^{n_2})\rightarrow
L^2({\mathbb R}^{n_2}) }dx_1\leq C\gamma_1^{-\delta},
\end{eqnarray}

\begin{eqnarray} \label{e6.15}
 \int_{|x_2-y_2|>\gamma_2t_2}
\|{\widetilde K}^{(2)}_{F({L_1},
{L_2})(I-e^{-t_2^2L_2})}{(x_2,y_2)}\|_{ L^2({\mathbb
R}^{n_1})\rightarrow L^2({\mathbb R}^{n_1}) }dx_2\leq
C\gamma_2^{-\delta},
\end{eqnarray}

\noindent
and
\begin{eqnarray}  \label{e6.16}
 \int_{\substack{|x_1-y_1|>\gamma_1t_1\\ |x_2-y_2|>\gamma_2t_2}}\big|K_{F({L_1},{L_2})(I-e^{-t^2_1L_1})(I-e^{-t^2_2L_2})}(x_1, y_1, x_2, y_2)\big|dx_1dx_2\leq
C\gamma_1^{-\delta}\gamma_2^{-\delta}.
\end{eqnarray}

\medskip

Let us prove (\ref{e6.14}).
We write
$F(\lambda_1,\lambda_2)=F(\lambda_1,\lambda_2)-F(0,0)+F(0,0)$.
Replacing $F$ by $F-F(0,0)$, we may assume in the sequel that
$F(0,0)=0$. Then we choose a function $\omega$ in $C^{\infty}_c({\mathbb R}_+)$ supported in $[1/4, 1]$ such that
 $\sum_{\ell\in{\mathbb Z}}\omega(2^\ell \lambda)=1, \ \forall \lambda\in {\mathbb R}_+.$
  Then we
have for all $\lambda_1,\lambda_2\geq 0$,

\begin{eqnarray}
&&F(\lambda_1,\lambda_2)=\sum_{\ell=-\infty}^\infty
\omega(2^{-\ell}\lambda_1)F(\lambda_1,\lambda_2)=:
\sum_{\ell=-\infty}^\infty \omega_{(\ell, 1)}F (\lambda_1,\lambda_2).\nonumber
\end{eqnarray}

\smallskip

\noindent Let $\Phi_{(t_1, 0)}(\lambda_1)=(1-e^{-(t_1\lambda_1)^2})$. Then we have

\begin{eqnarray*}
&&F({L_1}, {L_2})(1 -e^{-t_1^2{L_1}}) =\sum_{\ell=-\infty}^\infty
\omega_{(\ell, 1)} F  (1-\Phi_{(t_1, 0)})({L_1}, {L_2}) .\nonumber
\end{eqnarray*}

\noindent   By Proposition \ref{prop6.2},
we have
\begin{eqnarray*}
&&\hspace{-0.5cm}\int_{|x_1-y_1|>\gamma_1t_1} \|{\widetilde
K}^{(1)}_{\omega_{(\ell, 1)} F  (1-\Phi_{(t_1, 0)})({L_1},
{L_2})}{(x_1,y_1)}\|_{ L^2({\mathbb
R}^{n_2})\rightarrow L^2({\mathbb R}^{n_2}) }dx_1\\
&&\hskip.2cm \leq  \bigg( \int_{|x_1-y_1|>\gamma_1t_1} \|{\widetilde
K}^{(1)}_{\omega_{(\ell, 1)} F  (1-\Phi_{(t_1, 0)})({L_1},
{L_2})}{(x_1,y_1)}\|_{ L^2({\mathbb R}^{n_2})\rightarrow
L^2({\mathbb R}^{n_2}) }^2(1+2^{\ell}|x_1-y_1|)^{2s_1}dx_1
\bigg)^{1/2}\\
&&\hskip1cm
\times\bigg(\int_{{\mathbb R}^{n_1}}(1+2^{\ell}|x_1-y_1|)^{-2s_1}dx_1
\bigg)^{1/2}\\
&&\hskip.2cm \leq C \gamma_1^{-\eta/2}    \min\big\{\big(2^{\ell}t_1 \big)^{-\eta/2}, 2^{\ell}t_1
\big\}
  \|\delta_{(2^\ell, 1)}[\omega_{(\ell, 1)} F   ]\|_{W^{2}_{s_1+{1+\epsilon\over 2},{s_2\over 2}}}.
\end{eqnarray*}

\noindent Therefore,

\begin{eqnarray*}
&&\hspace{-1.2cm}\int_{|x_1-y_1|>\gamma_1t_1} \|{\widetilde
K}^{(1)}_{F({L_1}, {L_2})(I-e^{-t_1^2 {L_1}})}{(x_1,y_1)}\|_{
L^2({\mathbb
R}^{n_2})\rightarrow L^2({\mathbb R}^{n_2}) }dx_1\\
 &\leq& C\gamma_1^{-\eta/2}
\sum_{\ell }
\min\big\{\big(2^{\ell}t_1 \big)^{-\eta/2},   2^{\ell}t_1  \big\}
  \|\delta_{(2^\ell, 1)}[\omega_{(\ell, 1)} F   ]\|_{W^{2}_{s_1+{1+\epsilon\over 2},{s_2\over 2}}}
 \\
&\leq& C
 \gamma_1^{-\eta/2} \sup_{\ell}\|\delta_{(2^\ell, 1)}[\omega_{(\ell, 1)} F   ]\|_{W^{2}_{s_1+{1+\epsilon\over 2},{s_2\over 2}}}
 \end{eqnarray*}

\noindent as required to prove estimate (\ref{e6.14}).

From Proposition~\ref{prop6.2},
the similar argument  as above shows  (\ref{e6.15}) and (\ref{e6.16}), and so we skip
 it. This completes the proof of Theorem 5.2.
 \end{proof}

\bigskip

{\bf Acknowledgments.} 
X.T. Duong was
supported by Australia Research Council (ARC).
 L.X. Yan was supported by   Australia Research Council (ARC) and
 NNSF of China (Grant No.  10925106),
 and Grant for Senior Scholars from the Association of Colleges and Universities of Guangdong.

\end{document}